\documentclass[a4paper,10pt]{article}
\usepackage[utf8]{inputenc}
\usepackage{amsmath,amssymb,amsthm}
\usepackage[colorlinks,citecolor=blue]{hyperref}
\usepackage{cleveref}
\usepackage{enumitem}
\usepackage{mathrsfs}
\usepackage{graphicx}
\usepackage{tikz}
\usetikzlibrary{patterns}

\newcommand{\ZZ}{\mathbb{Z}}

\newcommand{\ov}[1]{\overline{#1}}

\newtheorem{theorem}{Theorem}[section] 
\newtheorem{proposition}[theorem]{Proposition} 
\newtheorem{conjecture}[theorem]{Conjecture} 
 
\newtheorem{lemma}[theorem]{Lemma}

\newtheorem{remark}[theorem]{Remark}

\newcommand{\oeis}[1]{\href{https://oeis.org/#1}{{#1}}}

\newcommand{\defi}[1]{\textit{\color{red}{#1}}}
\newcommand{\level}{\mathscr{L}}
\newcommand{\Dyck}{\mathscr{D}}
\newcommand{\mono}{\overline{\mathbb{M}}_1}
\newcommand{\monod}{\mathbb{M}_1}
\newcommand{\monoi}{\mathbb{M}_2}
\newcommand{\sm}{\leq}
\newcommand{\wmin}{w_{\min}}
\newcommand{\pscolle}{\ast}
\newcommand{\colle}{\,\sharp\,}
\newcommand{\up}{\operatorname{Up}}
\newcommand{\deri}{\mathsf{D}}
\newcommand{\des}{\operatorname{Desc}}
\newcommand{\rise}{\operatorname{Rise}}

\title{Some properties of a new partial order\\ on Dyck paths}
\author{F. Chapoton}
\date{\today}

\begin{document}

\maketitle

\section*{Introduction}

This article defines a new partial order on extremely classical
combinatorial objects, namely Dyck paths, and describes some of its
basic or subtle properties.

Let us start by explaining in some detail how this partial order was
discovered. The starting point was the study of the intervals in the
Tamari lattice. This was initiated in \cite{chap_slc}, where these
intervals were enumerated. One interesting point made there was the fact
that the number of Tamari intervals of size $n$ is exactly the number
of triangulations of size $n$, found by Tutte in one of his
foundational articles on planar maps \cite{tutte}. It was only
understood much later that the set of Tamari intervals admits a
natural and interesting partial order. Indeed, the set of intervals in
any poset can be ordered by the relation $[a,b] \leq [c,d]$ if and
only if $a \leq c$ and $b \leq d$. When drawing carefully the Hasse
diagram for these posets of Tamari intervals, one obtains a picture
which apparently has been considered previously from another point of
view. This other context is the study of homotopy associativity and
more precisely the cellular diagonal for the associahedra that
is required to define the tensor product of $A_{\infty}$ algebras
\cite{loday}.

By admitting that this is really the same picture, one can consider
that Tamari intervals can be gathered into higher-dimensional cells,
and that every such cell has unique top and bottom elements. There is
a distinguished top cell, made of all intervals whose maximum is the
unique maximum of the Tamari lattice. Elements in this top cell are
indexed by planar binary trees. It appears that every Tamari interval
in this top cell is the top element of some unique cell. Taking the
bottom elements of these cells gives another subset indexed by planar
binary trees. The partial order that we study in the present article
is nothing else than the order induced on this subset of the poset of
all Tamari intervals. Note that this description is just the original
motivation, but instead we will start here from a purely combinatorial
definition and will not prove that it coincides with this former description.

\smallskip

Let us now present the contents of the article. The first section
contains the definition of the posets $\Dyck_n$ for $n \geq 0$,
starting from a directed graph that turns out to be their Hasse
diagram. Some basic properties are given, including their number of
maximal elements and their relationship to the Tamari lattices.

Then come several sections that prepare the ground for the first main
result, by a careful combinatorial description of the intervals in
$\Dyck_n$.

Sections 2 and 3 present a structure of graded monoid on the disjoint
union of all elements of $\Dyck_n$ for $n \geq 1$, and some
compatibility of the monoid product with the partial orders. This is
then built upon in the section 4 to define a graded monoid structure
on the disjoint union of all sets of intervals in $\Dyck_n$ for
$n \geq 1$. All these monoids are shown to be free. Section 4 also
contains a result to the effect that some principal upper ideals are
products of smaller principal upper ideals.

Section 5 is the last piece of the combinatorial puzzle, namely the
study of a specific subset of core intervals in $\Dyck_{n+2}$ and the
construction of a bijection with intervals in $\Dyck_{n}$ together
with a catalytic data.

All this combinatorial preparation is then used in section 6 to prove
the first main result: the number of intervals in $\Dyck_n$ is also
given by another formula of Tutte, namely his formula for the number
of number of rooted bicubic planar maps. This is achieved by deducing
a functional equation from the combinatorial description, and then
solving this catalytic equation. This enumerative result is yet
another instance of the growing connection between Tamari and related
posets and various kinds of maps, see for example \cite{fangpr,fang2,fang3}.

Section 7 proves the second main result: all the posets $\Dyck_n$ are meet-semilattices,
which means that any two elements have a greatest lower bound. This
requires a somewhat technical study of properties in $\Dyck_n$ of
Dyck paths that share a common prefix of some length. The proof also gives
an algorithm for computing the meet.

Section 8 briefly evoques another interesting aspect of the posets
$\Dyck_n$, namely the probable existence of many derived equivalences
between some of their intervals. This certainly deserves further
investigation.

Section 9 tells a surprising story, namely an unexpected
connection with another familly of objects called the Hochschild
polytopes, maybe not so well-known, that appeared in the works of
Saneblidze in algebraic topology \cite{san_note, san_arxiv}. One
proves that there is a bijection between elements in a specific
interval $F_n$ of $\Dyck_{n+2}$ and the vertices of the cell complexes
of Saneblidze that realize the Hochschild polytopes. This bijection is
probably also an isomorphism of posets, which would prove that the
order of Saneblidze is a lattice. This question is left open here, for
lack of a strong enough motivation.

The final section 10 contains one result and various remarks. The
result is an unexpected symmetry for some kind of $h$-polynomial that
enumerates vertices according to a coloring of the covering
relations in $\Dyck_n$.

Section 11 is an appendix recalling some classical properties of Tamari lattices.

\section{Construction and first properties}

A \defi{Dyck path} of size $n \geq 0$ is a lattice path from $(0,0)$ to
$(2 n,0)$ using only north-east and south-east steps and staying weakly above
the horizontal line.

We will also consider Dyck paths of size $n$ as words of length $2 n$
in the alphabet $\{0,1\}$, where $1$ stands for a north-east step and
$0$ for a south-east step.

The next figure is a typical example.
\begin{figure}[h]
\centering
\begin{tikzpicture}[scale=0.2]
\draw[gray,very thin] (0, 0) grid (48, 8);  
\draw[rounded corners=1, color=black, line width=2] (0, 0) -- (1, 1) -- (2, 2) -- (3, 3) -- (4, 4) -- (5, 3) -- (6, 4) -- (7, 5) -- (8, 4) -- (9, 3) -- (10, 4) -- (11, 5) -- (12, 4) -- (13, 3) -- (14, 2) -- (15, 1) -- (16, 2) -- (17, 1) -- (18, 0) -- (19, 1) -- (20, 2) -- (21, 1) -- (22, 2) -- (23, 3) -- (24, 2) -- (25, 3) -- (26, 4) -- (27, 3) -- (28, 4) -- (29, 5) -- (30, 6) -- (31, 5) -- (32, 6) -- (33, 5) -- (34, 6) -- (35, 7) -- (36, 8) -- (37, 7) -- (38, 6) -- (39, 5) -- (40, 4) -- (41, 5) -- (42, 4) -- (43, 3) -- (44, 2) -- (45, 1) -- (46, 0) -- (47, 1) -- (48, 0);
\end{tikzpicture}
\end{figure}

The \defi{area} under a Dyck path is the surface of the domain between
the horizontal line and the Dyck path.

Every Dyck path can be uniquely written as the concatenation of
several blocks, where in every block the only vertices on the
horizontal line are the first and last ones. If there is only one
block, the Dyck path is said to be \defi{block-indecomposable}. The
example displayed above has $3$ blocks.

Inside a Dyck path, a subsequence of consecutive steps is called a
\defi{subpath} if it starts and ends at the same height and keeps
strictly above that height in between. Here the height is the
vertical coordinate, increased by north-east steps and decreased by
south-east steps.

Let us say that a subpath $x$ of a Dyck path $w$ is \defi{movable} if
it is preceded in $w$ by the letter $0$ and
\begin{itemize}
\item either $x$ ends at the last letter of $w$,
\item or $x$ is followed in $w$ by the letter $1$.
\end{itemize}

\begin{figure}[h]
\centering
\begin{minipage}[b]{0.3\textwidth}
\begin{tikzpicture}[scale=0.35]
  \draw[gray,very thin] (0, 0) grid (10, 2);
  \draw[rounded corners=1, color=black, line width=2] (0, 0) -- (1, 1) -- (2, 0) -- (3, 1) -- (4, 2) -- (5, 1) -- (6, 2) -- (7, 1) -- (8, 0) -- (9, 1) -- (10, 0);
  \draw[rounded corners=1, pattern=north west lines, pattern color=teal] (5, 1) -- (6, 2)-- (7, 1) ;
\end{tikzpicture}\caption{A subpath which is not movable.}\end{minipage}\hfill
\begin{minipage}[b]{0.3\textwidth}
\begin{tikzpicture}[scale=0.35]
  \draw[gray,very thin] (0, 0) grid (10, 2);
  \draw[rounded corners=1, color=black, line width=2] (0, 0) -- (1, 1) -- (2, 0) -- (3, 1) -- (4, 2) -- (5, 1) -- (6, 2) -- (7, 1) -- (8, 0) -- (9, 1) -- (10, 0);
  \draw[rounded corners=1, pattern=north west lines, pattern color=teal] (2, 0) -- (3, 1) -- (4, 2) -- (5, 1) -- (6, 2) -- (7, 1) -- (8, 0);
\end{tikzpicture}\caption{A movable subpath.}\end{minipage}\hfill
\begin{minipage}[b]{0.3\textwidth}
\begin{tikzpicture}[scale=0.35]
  \draw[gray,very thin] (0, 0) grid (10, 2);
  \draw[rounded corners=1, color=black, line width=2] (0, 0) -- (1, 1) -- (2, 0) -- (3, 1) -- (4, 2) -- (5, 1) -- (6, 2) -- (7, 1) -- (8, 0) -- (9, 1) -- (10, 0);
  \draw[rounded corners=1, pattern=north west lines, pattern color=teal] (8, 0) -- (9, 1) -- (10, 0);
\end{tikzpicture}\caption{A movable subpath.}\end{minipage}
\end{figure}

For every Dyck path $w$ and every movable subpath $x$ in $w$, let
$N(w,x)$ be the number of consecutive $0$ letters that appear just
before $x$. For any integer $1 \leq i \leq N(w,x)$ (corresponding to a
choice among the consecutive $0$ letters just before $x$), let us
define another Dyck path $M(w,x,i)$ as the following word:
\begin{itemize}
\item first the initial part of $w$ until the letter before the chosen $0$,
\item then $x$,
\item then the $0$ letters starting at the chosen $0$,
\item then the final part of $w$ after $x$.
\end{itemize}
Graphically, this transformation corresponds to sliding the
subpath $x$ in the north-west direction by one or several steps.

\smallskip

\begin{figure}
\centering
\begin{tikzpicture}[scale=0.4]  
\draw[gray,very thin] (0, 0) grid (12, 3);  
\draw[rounded corners=1, color=black, line width=2] (0, 0) -- (1, 1) -- (2, 2) -- (3, 3) -- (4, 2) -- (5, 3) -- (6, 2) -- (7, 1) -- (8, 0) -- (9, 1) -- (10, 2) -- (11, 1) -- (12, 0);
\draw[rounded corners=1, pattern=north west lines, pattern color=teal] 
(8, 0) -- (9, 1) -- (10, 2) -- (11, 1) -- (12, 0);
\end{tikzpicture}\hfill$\longrightarrow$\hfill
\begin{tikzpicture}[scale=0.4]  
\draw[gray,very thin] (0, 0) grid (12, 4);  
\draw[rounded corners=1, color=black, line width=2] (0, 0) -- (1, 1) -- (2, 2) -- (3, 3) -- (4, 2) -- (5, 3) -- (6, 2) -- (7, 3) -- (8, 4) -- (9, 3) -- (10, 2) -- (11, 1) -- (12, 0);
\draw[rounded corners=1, pattern=north west lines, pattern color=teal] 
(6, 2) -- (7, 3) -- (8, 4)-- (9, 3) -- (10, 2);
\end{tikzpicture}
\caption{A covering relation $w \to w'$}
\end{figure}

Let $\Dyck_n$ be the set of Dyck paths of size $n$. Let us
introduce a directed graph $\Gamma_n$ with vertex set $\Dyck_n$. It
has edges from every Dyck path $w$ to all Dyck paths $M(w,x,i)$ where
$x$ is a movable subpath of $w$ and $i$ an integer between $1$ and
$N(w,x)$.

Let $\wmin$ be the unique Dyck path made of
alternating $1$ and $0$.
\begin{proposition}
  The directed graph $\Gamma_n$ is connected and acyclic with $\wmin$ as
  unique source element.
\end{proposition}
\begin{proof}
  First, every edge in $\Gamma_n$ increases strictly the area under
  the Dyck path, so there cannot be any oriented cycles.

  Let $w$ be any Dyck path not equal to $\wmin$. Then $w$ has a
  least one block $x$ of maximal height at least $2$. This block ends by a
  sequence of at least two $0$ preceded by a $1$. Let $w'$ be the Dyck
  path obtained by exchanging this $1$ and that sequence of $0$ except
  the last one. Then $w'$ has a strictly smaller area than $w$
  and there is an edge in $\Gamma_n$ from $w'$ to $w$.

  Therefore, by induction on the area, for every Dyck path $w$, there
  exists a path in $\Gamma_n$ from $\wmin$ to $w$.
\end{proof}

\begin{proposition}
  The directed graph $\Gamma_n$ is the Hasse diagram of a partial order.
\end{proposition}
\begin{proof}
  This means that $\Gamma_n$ is acyclic and transitively reduced. It
  remains only to prove the second property.
  
  So let us consider an edge $w \to M(w,x,i)$ and assume that one can
  find a sequence (S) of edges $w \to M(w,x',i') \to \dots \to M(w,x,i)$.

  Every edge in $\Gamma_n$ can be considered as a sequence of several
  moves where a subpath is slided by just one step in the north-west
  direction. As recalled in \S \ref{tamari}, these moves are exactly
  the cover relations in the Tamari partial order on the same set of Dyck
  paths.

  Therefore, one can refine the sequence (S) of edges in $\Gamma_n$
  into a sequence of cover moves in the Tamari lattice. By
  \cref{tamari_chain}, the interval in the Tamari lattice between $w$
  and $M(w,x,i)$ is just a chain obtained by repeated sliding of the subpath
  $x$. This implies that all the intermediate vertices of the
  sequence (S) are obtained by sliding the same subpath $x$.

  But in the graph $\Gamma_n$, there cannot be any two consecutive
  edges where exactly the same subpath is slided. Indeed, after being
  slided once, the subpath is followed by a letter $0$, hence no
  longer movable. It follows that the sequence (S) is reduced to the
  single edge $w \to M(w,x,i)$.
\end{proof}

Let us denote by $\sm$ the partial order relation on $\Dyck_n$ thus defined. The edges of $\Gamma_n$ are now understood as the cover relations in the poset $(\Dyck_n, \sm)$.

\medskip

We propose to call this partial order the \defi{dexter order} on Dyck
paths. This choice of terminology is motivated by the fact that a
symmetry is broken when compared to the Tamari lattice. This is also
linked to some pattern-exclusion in interval-posets for the Tamari
lattices, which reminds of the chirality of seashells.

\begin{figure}[h]
  \centering
  \begin{tikzpicture}[>=latex,line join=bevel,scale=0.3]
\node (node_2) at (876.0bp,107.0bp) [draw,draw=none] {$\vcenter{\hbox{$\begin{tikzpicture}[scale=0.15]  \draw[gray,very thin] (0, 0) grid (8, 2);  \draw[rounded corners=1, color=black, line width=1.5] (0, 0) -- (1, 1) -- (2, 0) -- (3, 1) -- (4, 2) -- (5, 1) -- (6, 0) -- (7, 1) -- (8, 0);\end{tikzpicture}$}}$};
  \node (node_3) at (1129.0bp,223.0bp) [draw,draw=none] {$\vcenter{\hbox{$\begin{tikzpicture}[scale=0.15]  \draw[gray,very thin] (0, 0) grid (8, 2);  \draw[rounded corners=1, color=black, line width=1.5] (0, 0) -- (1, 1) -- (2, 0) -- (3, 1) -- (4, 2) -- (5, 1) -- (6, 2) -- (7, 1) -- (8, 0);\end{tikzpicture}$}}$};
  \node (node_9) at (370.0bp,367.0bp) [draw,draw=none] {$\vcenter{\hbox{$\begin{tikzpicture}[scale=0.15]  \draw[gray,very thin] (0, 0) grid (8, 3);  \draw[rounded corners=1, color=black, line width=1.5] (0, 0) -- (1, 1) -- (2, 2) -- (3, 1) -- (4, 2) -- (5, 3) -- (6, 2) -- (7, 1) -- (8, 0);\end{tikzpicture}$}}$};
  \node (node_8) at (117.0bp,367.0bp) [draw,draw=none] {$\vcenter{\hbox{$\begin{tikzpicture}[scale=0.15]  \draw[gray,very thin] (0, 0) grid (8, 2);  \draw[rounded corners=1, color=black, line width=1.5] (0, 0) -- (1, 1) -- (2, 2) -- (3, 1) -- (4, 2) -- (5, 1) -- (6, 2) -- (7, 1) -- (8, 0);\end{tikzpicture}$}}$};
  \node (node_7) at (117.0bp,223.0bp) [draw,draw=none] {$\vcenter{\hbox{$\begin{tikzpicture}[scale=0.15]  \draw[gray,very thin] (0, 0) grid (8, 2);  \draw[rounded corners=1, color=black, line width=1.5] (0, 0) -- (1, 1) -- (2, 2) -- (3, 1) -- (4, 2) -- (5, 1) -- (6, 0) -- (7, 1) -- (8, 0);\end{tikzpicture}$}}$};
  \node (node_6) at (370.0bp,223.0bp) [draw,draw=none] {$\vcenter{\hbox{$\begin{tikzpicture}[scale=0.15]  \draw[gray,very thin] (0, 0) grid (8, 2);  \draw[rounded corners=1, color=black, line width=1.5] (0, 0) -- (1, 1) -- (2, 2) -- (3, 1) -- (4, 0) -- (5, 1) -- (6, 2) -- (7, 1) -- (8, 0);\end{tikzpicture}$}}$};
  \node (node_5) at (370.0bp,107.0bp) [draw,draw=none] {$\vcenter{\hbox{$\begin{tikzpicture}[scale=0.15]  \draw[gray,very thin] (0, 0) grid (8, 2);  \draw[rounded corners=1, color=black, line width=1.5] (0, 0) -- (1, 1) -- (2, 2) -- (3, 1) -- (4, 0) -- (5, 1) -- (6, 0) -- (7, 1) -- (8, 0);\end{tikzpicture}$}}$};
  \node (node_4) at (623.0bp,223.0bp) [draw,draw=none] {$\vcenter{\hbox{$\begin{tikzpicture}[scale=0.15]  \draw[gray,very thin] (0, 0) grid (8, 3);  \draw[rounded corners=1, color=black, line width=1.5] (0, 0) -- (1, 1) -- (2, 0) -- (3, 1) -- (4, 2) -- (5, 3) -- (6, 2) -- (7, 1) -- (8, 0);\end{tikzpicture}$}}$};
  \node (node_13) at (623.0bp,367.0bp) [draw,draw=none] {$\vcenter{\hbox{$\begin{tikzpicture}[scale=0.15]  \draw[gray,very thin] (0, 0) grid (8, 4);  \draw[rounded corners=1, color=black, line width=1.5] (0, 0) -- (1, 1) -- (2, 2) -- (3, 3) -- (4, 4) -- (5, 3) -- (6, 2) -- (7, 1) -- (8, 0);\end{tikzpicture}$}}$};
  \node (node_12) at (1048.0bp,367.0bp) [draw,draw=none] {$\vcenter{\hbox{$\begin{tikzpicture}[scale=0.15]  \draw[gray,very thin] (0, 0) grid (8, 3);  \draw[rounded corners=1, color=black, line width=1.5] (0, 0) -- (1, 1) -- (2, 2) -- (3, 3) -- (4, 2) -- (5, 3) -- (6, 2) -- (7, 1) -- (8, 0);\end{tikzpicture}$}}$};
  \node (node_1) at (623.0bp,107.0bp) [draw,draw=none] {$\vcenter{\hbox{$\begin{tikzpicture}[scale=0.15]  \draw[gray,very thin] (0, 0) grid (8, 2);  \draw[rounded corners=1, color=black, line width=1.5] (0, 0) -- (1, 1) -- (2, 0) -- (3, 1) -- (4, 0) -- (5, 1) -- (6, 2) -- (7, 1) -- (8, 0);\end{tikzpicture}$}}$};
  \node (node_10) at (876.0bp,223.0bp) [draw,draw=none] {$\vcenter{\hbox{$\begin{tikzpicture}[scale=0.15]  \draw[gray,very thin] (0, 0) grid (8, 3);  \draw[rounded corners=1, color=black, line width=1.5] (0, 0) -- (1, 1) -- (2, 2) -- (3, 3) -- (4, 2) -- (5, 1) -- (6, 0) -- (7, 1) -- (8, 0);\end{tikzpicture}$}}$};
  \node (node_11) at (469.0bp,511.0bp) [draw,draw=none] {$\vcenter{\hbox{$\begin{tikzpicture}[scale=0.15]  \draw[gray,very thin] (0, 0) grid (8, 3);  \draw[rounded corners=1, color=black, line width=1.5] (0, 0) -- (1, 1) -- (2, 2) -- (3, 3) -- (4, 2) -- (5, 1) -- (6, 2) -- (7, 1) -- (8, 0);\end{tikzpicture}$}}$};
  \node (node_0) at (623.0bp,19.0bp) [draw,draw=none] {$\vcenter{\hbox{$\begin{tikzpicture}[scale=0.15]  \draw[gray,very thin] (0, 0) grid (8, 1);  \draw[rounded corners=1, color=black, line width=1.5] (0, 0) -- (1, 1) -- (2, 0) -- (3, 1) -- (4, 0) -- (5, 1) -- (6, 0) -- (7, 1) -- (8, 0);\end{tikzpicture}$}}$};
  \draw [red,->] (node_1) ..controls (623.0bp,147.89bp) and (623.0bp,156.89bp)  .. (node_4);
  \draw [blue,->,dashed] (node_5) ..controls (267.78bp,154.06bp) and (229.62bp,171.25bp)  .. (node_7);
  \draw [blue,->,dashed] (node_2) ..controls (978.22bp,154.06bp) and (1016.4bp,171.25bp)  .. (node_3);
  \draw [red,->] (node_0) ..controls (703.03bp,47.204bp) and (739.52bp,59.607bp)  .. (node_2);
  \draw [red,->] (node_3) ..controls (1101.7bp,271.81bp) and (1090.1bp,292.17bp)  .. (node_12);
  \draw [red,->] (node_4) ..controls (623.0bp,278.31bp) and (623.0bp,287.03bp)  .. (node_13);
  \draw [red,->] (node_2) ..controls (876.0bp,147.89bp) and (876.0bp,156.89bp)  .. (node_10);
  \draw [blue,->,dashed] (node_10) ..controls (839.16bp,316.73bp) and (802.66bp,387.26bp)  .. (749.0bp,428.0bp) .. controls (704.94bp,461.45bp) and (647.52bp,481.29bp)  .. (node_11);
  \draw [red,->] (node_8) ..controls (200.14bp,409.12bp) and (222.33bp,419.35bp)  .. (243.0bp,428.0bp) .. controls (274.8bp,441.31bp) and (309.61bp,454.56bp)  .. (node_11);
  \draw [red,->] (node_0) ..controls (542.97bp,47.204bp) and (506.48bp,59.607bp)  .. (node_5);
  \draw [red,->] (node_5) ..controls (490.21bp,138.84bp) and (493.12bp,139.43bp)  .. (496.0bp,140.0bp) .. controls (604.27bp,161.39bp) and (636.96bp,149.25bp)  .. (node_10);
  \draw [blue,->,dashed] (node_10) ..controls (948.6bp,283.94bp) and (967.11bp,299.22bp)  .. (node_12);
  \draw [red,->] (node_6) ..controls (449.99bp,268.9bp) and (478.72bp,285.02bp)  .. (node_13);
  \draw [blue,->,dashed] (node_7) ..controls (117.0bp,275.69bp) and (117.0bp,302.1bp)  .. (node_8);
  \draw [red,->] (node_10) ..controls (776.3bp,279.96bp) and (757.43bp,290.55bp)  .. (node_13);
  \draw [red,->] (node_7) ..controls (204.15bp,272.91bp) and (243.42bp,294.95bp)  .. (node_9);
  \draw [red,->] (node_1) ..controls (520.78bp,154.06bp) and (482.62bp,171.25bp)  .. (node_6);
  \draw [blue,->,dashed] (node_6) ..controls (370.0bp,271.44bp) and (370.0bp,291.25bp)  .. (node_9);
  \draw [red,->] (node_0) ..controls (623.0bp,45.551bp) and (623.0bp,54.96bp)  .. (node_1);
  \draw [red,->] (node_5) ..controls (370.0bp,152.23bp) and (370.0bp,166.73bp)  .. (node_6);
  \draw [red,->] (node_2) ..controls (783.82bp,149.54bp) and (758.9bp,160.76bp)  .. (node_4);
\end{tikzpicture}
  \caption{Hasse diagram of the poset $(\Dyck_4, \leq)$. Edge colors will be explained and used in section \ref{hpoly}.}
  \label{fig_dyck4}
\end{figure}
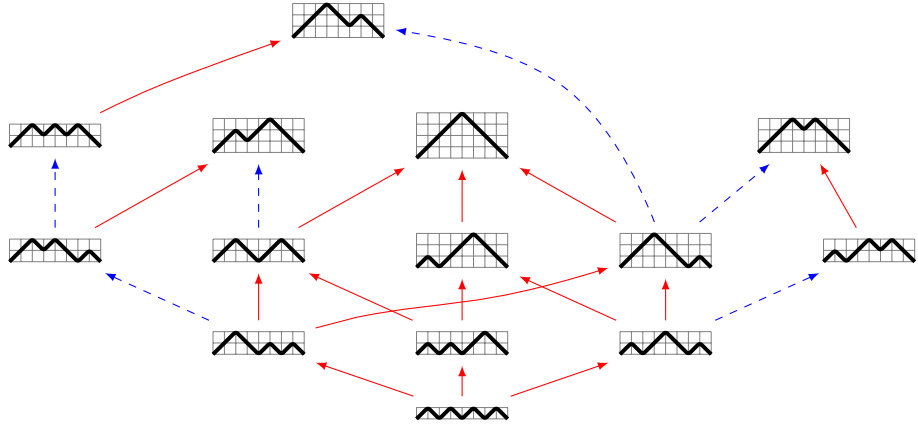

\subsection{First properties}

\begin{proposition}
  The maximal elements of $(\Dyck_n, \sm)$ are exactly the
  block-indecomposable Dyck paths that do not contain any subpath that is
  both preceded by $0$ and followed by $1$.
\end{proposition}
\begin{proof}
  The property of $w$ being maximal is equivalent to the non-existence
  of movable subpaths in $w$.

  According to their definition, movable subpaths can exist in two
  distinct ways. The first way is when the movable subpath is the last
  block of $w$. This happens if and only if $w$ contains at least two
  blocks. The second way is the existence of a subpath preceded by
  $0$ and followed by $1$.

  Therefore being maximal is equivalent to being block-indecomposable and
  not having this second kind of subpaths.
\end{proof}

\begin{proposition}
  The sets of maximal elements are counted by the Motzkin numbers (\oeis{A1006}).
\end{proposition}
\begin{proof}
  Removing the initial $1$ and the final $0$ defines a bijection from
  the set of maximal elements to the set $\mathscr{M}$ of Dyck paths
  that do not contain any subpath that is both preceded by $0$ and
  followed by $1$. Then the decomposition into blocks of such a path
  has at most two blocks and these blocks satisfy the same
  condition. One therefore gets
  \begin{equation}
    M = 1 + t M + t^2 M^2 =  1 + t + 2 t^2 + 4 t^3 + 9 t^4 + \dots
  \end{equation}
  for the generating series $M$ of the set $\mathscr{M}$. This is the usual
  equation for the generating series of Motzkin numbers.
\end{proof}

Let us now give examples of intervals in $\Dyck_n$ where some
properties do not hold. The interval between $\wmin$ and the Dyck path
$(1, 1, 1, 1, 0, 0, 1, 0, 0, 1, 0, 0)$ is not semi-distributive, and
therefore not congruence-uniform ; it is also not extremal. For these
notions, see general references on lattice theory such as
\cite{freese, gratzer}. The interval between $\wmin$ and the Dyck
path
\begin{equation*}
(1, 1, 1, 1, 1, 0, 0, 1, 0, 0, 0, 1, 0, 0)
\end{equation*}
has a Coxeter polynomial with some roots not on the unit circle (see
\cite{chap_cats, lenzing} for the representation-theoretic context).



\begin{remark}
  There are two simple natural inclusions of posets of $\Dyck_n$ in
  $\Dyck_{n+1}$, by concatenation of the Dyck path $(1,0)$ at the
  start or at the end.
\end{remark}

\subsection{Relations to other posets}

In this subsection, we prove that the partial order $(\Dyck_n, \leq)$
stands somewhere between the Tamari partial order \cite{tamari_friedman, tamari_festschrift} and the less
well-known comb partial order introduced by Pallo in \cite{pallo}.

The Tamari lattice structure on Dyck paths is recalled in \S \ref{tamari}.
\begin{proposition}
  Let $w$ and $w'$ be two Dyck paths in $\Dyck_n$ such that $w \sm w'$.
  Then $w$ is smaller than $w'$ in the Tamari order.
\end{proposition}
\begin{proof}
  By transitivity, it is enough to prove this under the additional
  assumption that there is an edge $w \to w'$ in $\Gamma_n$. It is
  clear that such an edge can be performed using a sequence of several
  cover relations in the Tamari lattice, which are just made by
  sliding a subpath by one step in the north-west direction.
\end{proof}

The comb partial order (or left-arm rotation order) was initially
defined in \cite{pallo} as a partial order on the set of binary
trees. Its cover relations are a subset of the cover relations of the
Tamari lattice, namely left-to-right rotations $(ab)c \to a(bc)$ of
binary trees where the two rotated vertices are on the leftmost branch
of the first binary tree. Using the same bijection $\sigma$ as \cite{berbon}
and the proof of their Prop. 2.1, one can see that these restricted
rotations correspond precisely to Tamari cover moves in $\Dyck_n$
where the slided subpath is at height $0$.

To summarize, the cover moves in the comb partial order on $\Dyck_n$
are the Tamari cover moves where the slided subpath is at height $0$.
\begin{proposition}
  Let $w$ and $w'$ be two Dyck paths in $\Dyck_n$ such that $w$ is
  smaller than $w'$ in the comb partial order. Then $w \leq w'$.
\end{proposition}
\begin{proof}
  This holds because every cover move in the comb poset is also a
  cover move in $(\Dyck_n, \leq)$, as a subpath at height $0$ cannot be followed
  by the letter $0$.
\end{proof}

\begin{figure}[h]
  \centering
\begin{tikzpicture}[scale=0.35,>=latex,line join=bevel,]
\node (node_4) at (286.0bp,223.0bp) [draw,draw=none] {$\vcenter{\hbox{$\begin{tikzpicture}[scale=.15]  \draw[gray,very thin] (0, 0) grid (6, 3);  \draw[rounded corners=1, color=black, line width=2] (0, 0) -- (1, 1) -- (2, 2) -- (3, 3) -- (4, 2) -- (5, 1) -- (6, 0);\end{tikzpicture}$}}$};
  \node (node_3) at (89.0bp,223.0bp) [draw,draw=none] {$\vcenter{\hbox{$\begin{tikzpicture}[scale=.15]  \draw[gray,very thin] (0, 0) grid (6, 2);  \draw[rounded corners=1, color=black, line width=2] (0, 0) -- (1, 1) -- (2, 2) -- (3, 1) -- (4, 2) -- (5, 1) -- (6, 0);\end{tikzpicture}$}}$};
  \node (node_2) at (89.0bp,107.0bp) [draw,draw=none] {$\vcenter{\hbox{$\begin{tikzpicture}[scale=.15]  \draw[gray,very thin] (0, 0) grid (6, 2);  \draw[rounded corners=1, color=black, line width=2] (0, 0) -- (1, 1) -- (2, 2) -- (3, 1) -- (4, 0) -- (5, 1) -- (6, 0);\end{tikzpicture}$}}$};
  \node (node_1) at (286.0bp,107.0bp) [draw,draw=none] {$\vcenter{\hbox{$\begin{tikzpicture}[scale=.15]  \draw[gray,very thin] (0, 0) grid (6, 2);  \draw[rounded corners=1, color=black, line width=2] (0, 0) -- (1, 1) -- (2, 0) -- (3, 1) -- (4, 2) -- (5, 1) -- (6, 0);\end{tikzpicture}$}}$};
  \node (node_0) at (187.0bp,19.0bp) [draw,draw=none] {$\vcenter{\hbox{$\begin{tikzpicture}[scale=.15]  \draw[gray,very thin] (0, 0) grid (6, 1);  \draw[rounded corners=1, color=black, line width=2] (0, 0) -- (1, 1) -- (2, 0) -- (3, 1) -- (4, 0) -- (5, 1) -- (6, 0);\end{tikzpicture}$}}$};
  \draw [black,->] (node_1) ..controls (286.0bp,147.89bp) and (286.0bp,156.89bp)  .. (node_4);
  \draw [black,->] (node_2) ..controls (89.0bp,152.23bp) and (89.0bp,166.73bp)  .. (node_3);
  \draw [black,->] (node_0) ..controls (156.85bp,46.456bp) and (144.63bp,57.185bp)  .. (node_2);
  \draw [black,->] (node_0) ..controls (217.46bp,46.456bp) and (229.81bp,57.185bp)  .. (node_1);
\end{tikzpicture}
\begin{tikzpicture}[scale=0.35,>=latex,line join=bevel,]
\node (node_4) at (286.0bp,223.0bp) [draw,draw=none] {$\vcenter{\hbox{$\begin{tikzpicture}[scale=.15]  \draw[gray,very thin] (0, 0) grid (6, 3);  \draw[rounded corners=1, color=black, line width=2] (0, 0) -- (1, 1) -- (2, 2) -- (3, 3) -- (4, 2) -- (5, 1) -- (6, 0);\end{tikzpicture}$}}$};
  \node (node_3) at (89.0bp,223.0bp) [draw,draw=none] {$\vcenter{\hbox{$\begin{tikzpicture}[scale=.15]  \draw[gray,very thin] (0, 0) grid (6, 2);  \draw[rounded corners=1, color=black, line width=2] (0, 0) -- (1, 1) -- (2, 2) -- (3, 1) -- (4, 2) -- (5, 1) -- (6, 0);\end{tikzpicture}$}}$};
  \node (node_2) at (89.0bp,107.0bp) [draw,draw=none] {$\vcenter{\hbox{$\begin{tikzpicture}[scale=.15]  \draw[gray,very thin] (0, 0) grid (6, 2);  \draw[rounded corners=1, color=black, line width=2] (0, 0) -- (1, 1) -- (2, 2) -- (3, 1) -- (4, 0) -- (5, 1) -- (6, 0);\end{tikzpicture}$}}$};
  \node (node_1) at (286.0bp,107.0bp) [draw,draw=none] {$\vcenter{\hbox{$\begin{tikzpicture}[scale=.15]  \draw[gray,very thin] (0, 0) grid (6, 2);  \draw[rounded corners=1, color=black, line width=2] (0, 0) -- (1, 1) -- (2, 0) -- (3, 1) -- (4, 2) -- (5, 1) -- (6, 0);\end{tikzpicture}$}}$};
  \node (node_0) at (187.0bp,19.0bp) [draw,draw=none] {$\vcenter{\hbox{$\begin{tikzpicture}[scale=.15]  \draw[gray,very thin] (0, 0) grid (6, 1);  \draw[rounded corners=1, color=black, line width=2] (0, 0) -- (1, 1) -- (2, 0) -- (3, 1) -- (4, 0) -- (5, 1) -- (6, 0);\end{tikzpicture}$}}$};
  \draw [red,->] (node_1) ..controls (286.0bp,147.89bp) and (286.0bp,156.89bp)  .. (node_4);
  \draw [blue,->,dashed] (node_2) ..controls (89.0bp,152.23bp) and (89.0bp,166.73bp)  .. (node_3);
  \draw [red,->] (node_0) ..controls (156.85bp,46.456bp) and (144.63bp,57.185bp)  .. (node_2);
  \draw [red,->] (node_0) ..controls (217.46bp,46.456bp) and (229.81bp,57.185bp)  .. (node_1);
  \draw [red,->] (node_2) ..controls (160.25bp,149.23bp) and (178.98bp,160.07bp)  .. (node_4);
\end{tikzpicture}
\begin{tikzpicture}[scale=0.35,>=latex,line join=bevel,]
\node (node_4) at (187.0bp,325.0bp) [draw,draw=none] {$\vcenter{\hbox{$\begin{tikzpicture}[scale=.15]  \draw[gray,very thin] (0, 0) grid (6, 3);  \draw[rounded corners=1, color=black, line width=2] (0, 0) -- (1, 1) -- (2, 2) -- (3, 3) -- (4, 2) -- (5, 1) -- (6, 0);\end{tikzpicture}$}}$};
  \node (node_3) at (89.0bp,209.0bp) [draw,draw=none] {$\vcenter{\hbox{$\begin{tikzpicture}[scale=.15]  \draw[gray,very thin] (0, 0) grid (6, 2);  \draw[rounded corners=1, color=black, line width=2] (0, 0) -- (1, 1) -- (2, 2) -- (3, 1) -- (4, 2) -- (5, 1) -- (6, 0);\end{tikzpicture}$}}$};
  \node (node_2) at (109.0bp,107.0bp) [draw,draw=none] {$\vcenter{\hbox{$\begin{tikzpicture}[scale=.15]  \draw[gray,very thin] (0, 0) grid (6, 2);  \draw[rounded corners=1, color=black, line width=2] (0, 0) -- (1, 1) -- (2, 2) -- (3, 1) -- (4, 0) -- (5, 1) -- (6, 0);\end{tikzpicture}$}}$};
  \node (node_1) at (286.0bp,209.0bp) [draw,draw=none] {$\vcenter{\hbox{$\begin{tikzpicture}[scale=.15]  \draw[gray,very thin] (0, 0) grid (6, 2);  \draw[rounded corners=1, color=black, line width=2] (0, 0) -- (1, 1) -- (2, 0) -- (3, 1) -- (4, 2) -- (5, 1) -- (6, 0);\end{tikzpicture}$}}$};
  \node (node_0) at (168.0bp,19.0bp) [draw,draw=none] {$\vcenter{\hbox{$\begin{tikzpicture}[scale=.15]  \draw[gray,very thin] (0, 0) grid (6, 1);  \draw[rounded corners=1, color=black, line width=2] (0, 0) -- (1, 1) -- (2, 0) -- (3, 1) -- (4, 0) -- (5, 1) -- (6, 0);\end{tikzpicture}$}}$};
  \draw [black,->] (node_1) ..controls (250.75bp,250.59bp) and (242.19bp,260.44bp)  .. (node_4);
  \draw [black,->] (node_2) ..controls (100.98bp,148.12bp) and (99.116bp,157.42bp)  .. (node_3);
  \draw [black,->] (node_0) ..controls (150.19bp,45.962bp) and (143.33bp,55.966bp)  .. (node_2);
  \draw [black,->] (node_0) ..controls (189.15bp,48.111bp) and (198.9bp,61.611bp)  .. (207.0bp,74.0bp) .. controls (227.12bp,104.77bp) and (248.17bp,140.7bp)  .. (node_1);
  \draw [black,->] (node_3) ..controls (123.9bp,250.59bp) and (132.36bp,260.44bp)  .. (node_4);
\end{tikzpicture}
\caption{Comparison of three partial orders: comb, dexter and Tamari}
\end{figure}
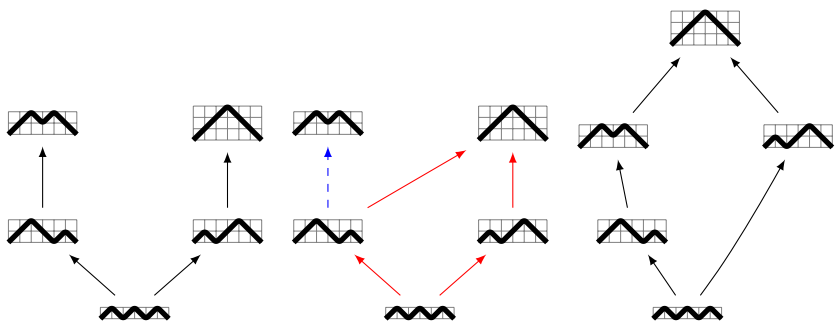

\section{A monoid on pseudo-Dyck paths}

Let us first introduce another kind of decomposition for
block-indecomposable Dyck paths.
\begin{proposition}
  \label{level_decomposition}
  Let $k\geq 0$ be an integer. Every block-indecomposable Dyck path ending
  with exactly $k + 1$ letters $0$ can be uniquely written as
  \begin{equation}
    \label{pieces}
    (1, w_1, 1, w_2, 1, w_3, \dots, 1, w_{k}, 1, 0^{k+1})
  \end{equation}
  for some Dyck paths $w_1, w_2, \dots, w_{k}$. This defines a
  bijection between block-indecomposable Dyck paths ending with exactly
  $k + 1$ letters $0$ and sequences of $k$ Dyck paths.
\end{proposition}
\begin{proof}
  Clearly, every Dyck path of the given shape \eqref{pieces} is
  block-indecomposable and ends with exactly $k+1$ letters $0$.

  Conversely, let us start with a block-indecomposable Dyck path $w$ ending
  with exactly $k + 1$ letters $0$. The aim is to recover all the
  $w_i$. Consider the word $v_k$ obtained from $w$ by removing the
  final $(1, 0^{k+1})$ part. The final height of $v_k$ is $k$. Let
  $w_{k}$ be the largest suffix of $v_k$ in which the height is greater
  than or equal to $k$. The letter before $w_{k}$ in $w$ must be
  $1$. Then one can define $v_{k-1}$ by removing $(1,w_k)$ at the end of
  $v_k$. Iterating this process of removing the largest suffix of
  height staying above the final height, one defines $w_1,\dots,w_k$
  that are all Dyck paths. This yields the desired decomposition,
  which is clearly unique.
\end{proof}
Let us call this decomposition the \defi{level-decomposition} of
$w$. We will write 
\begin{equation*}
  w = \level(w_1, w_2, \dots, w_{k})
\end{equation*}
to denote this situation.

In fact, both the decomposition into blocks and the
level-decomposition are special cases of products inside a monoid
structure. Let us describe this monoid.

\smallskip

It turns out to be more convenient to work (for a moment) with Dyck
paths minus their initial and final letters. Let us therefore define a
\defi{pseudo-Dyck path} to be a path obtained in this way from a
non-empty Dyck path. For a Dyck path $w$, let $\ov{w}$ denote the
pseudo-Dyck path obtained by removing the initial and final letters of
$w$.

Note that the height in a pseudo-Dyck path can reach the value $-1$,
taking $0$ as initial height.

\smallskip

Let us introduce a product $\pscolle$ on the set of pseudo-Dyck
paths. Let $u$ and $v$ be two such paths. Write $u = (u', 0^k)$ where
$u'$ is either empty (if $u$ itself is empty) or ends with the letter
$1$. Write $v = v' v''$ where $v'$ is the largest prefix where the
height is positive or zero. The product is
\begin{equation}
u \pscolle v = (u', v', 0^k, v'').
\end{equation}
This is again a pseudo-Dyck path, as it lies above the concatenation
$u v$, which is itself clearly a pseudo-Dyck path.

For example, one gets
\begin{equation*}
  (1,0) \pscolle (0,1) = (1, 0, 0, 1)
\end{equation*}
and
\begin{equation*}
  (0,1,1,0) \pscolle (1,0,1,0) = (0, 1, 1, 1, 0, 1, 0, 0).
\end{equation*}

\begin{lemma}
  The product $\pscolle$ is an associative product on the set of
  pseudo-Dyck paths.
\end{lemma}
\begin{proof}
  Consider three pseudo-Dyck paths $u$, $v$, $w$. Write $u = (u', 0^k)$,
  $w= w' w''$ and decompose also $v$ in both ways. One has to
  distinguish two cases, depending on whether $(1,v,0)$ is
  block-indecomposable or not. If it is, then $v = (v', 0^\ell)$ where $v'$ ends
  by the letter $1$. The largest prefix of positive height is $v$
  itself. In this case, the two ways to compute the product of $u,v,w$
  give $(u', v', w', 0^{\ell+k}, w'')$.

  If $(1,v,0)$ is not block-indecomposable, then $v = (v', v'', 0^\ell)$,
  where $v''$ ends with $1$ and $v'$ is the largest prefix of positive height.
  In this case, the two ways to compute the product of $u,v,w$
  give $(u', v', 0^k, v'', w', 0^{\ell}, w'')$. 
\end{proof}

This defines a monoid $\mono$ on the set of pseudo-Dyck paths. This
monoid is graded by the size, which is the number of letters $1$. Note
that the empty pseudo-Dyck path is the unit of this monoid.

\begin{lemma}
  \label{concat_colle}
  Let $u$ and $v$ be two Dyck paths. Then the product
  $\ov{u} \pscolle (0,1) \pscolle \ov{v}$ is the pseudo-Dyck path
  $\ov{u v}$ associated to the concatenation of $u$ and $v$.
\end{lemma}
\begin{proof}
  Using the definition of $\pscolle$, one can check that the pseudo-Dyck path
  $\ov{u} \pscolle (0,1) \pscolle \ov{v}$ is just the concatenation
  $(\ov{u},0,1,\ov{v})$. Hence the associated Dyck word is $(1,\ov{u},0,1,\ov{v},0)$, which is $u v$.
\end{proof}

\begin{lemma}
  \label{level_colle}
  Let $w$ be a block-indecomposable Dyck path with level-decomposition
  $w=\level(w_1,\dots,w_k)$. Then the pseudo-Dyck path $\ov{w}$ is 
  \begin{equation}
    \ov{w} = (w_1,1,0) \pscolle \dots \pscolle (w_k,1,0).
  \end{equation}
\end{lemma}
\begin{proof}
  By induction on $k$. This is true if $k=0$ for the empty
  level-decomposition of $(1,0)$ and the empty product. If $k=1$, the
  product has only one factor and the statement is also
  true. Otherwise, let $w'$ be the block-indecomposable Dyck path with
  level-decomposition $w'=\level(w_1,\dots,w_{k-1})$. Then one knows
  that $\ov{w'} = (w_1,1,0) \pscolle \dots \pscolle
  (w_{k-1},1,0)$. Computing the $\pscolle$ product of this with
  $(w_k,1,0)$, one gets exactly $\ov{w}$ on the left hand side and the
  expected product on the right hand side.
\end{proof}
This lemma also implies that the product $\ov{u} \pscolle \ov{v}$ for
two block-indecomposable Dyck paths $u$ and $v$ is $\ov{w}$, where $w$
is the block-indecomposable Dyck path whose level-decomposition is the
concatenation of the level-decompositions of $u$ and $v$.

\begin{lemma}
  \label{gen_M1}
  The monoid $\mono$ is generated by the elements $(w,1,0)$ where $w$
  runs over the set of Dyck paths, plus the element $(0,1)$.
\end{lemma}
\begin{proof}
  By the previous two lemmas, one can first write any pseudo-Dyck path as
  a $\pscolle$ product of some $\ov{w}$ for block-indecomposable Dyck
  paths $w$, with intermediate factors $(0,1)$. Then using the
  level-decomposition of each $w$, one can write $\ov{w}$ as a
  $\pscolle$ product of pseudo-Dyck paths of the shape $(w_i,1,0)$.
\end{proof}

\begin{proposition}
  \label{free_M1}
  The monoid $\mono$ is free on the
  generators $(w,1,0)$ where $w$ runs over the set of Dyck paths, plus
  the element $(0,1)$.
\end{proposition}
\begin{proof}
  Because the monoid is generated by these elements, it is enough to
  compare its generating series with the generating series of the free
  monoid on the same generators. This is a simple computation with the
  usual generating series of Catalan numbers.
\end{proof}

\section{The same monoid seen on Dyck paths}

Let us now go back to Dyck paths. Using the simple bijection
$w \to \ov{w}$ between non-empty Dyck paths and pseudo-Dyck paths, the
monoid structure $\mono$ can be transported to a monoid structure $\monod$ on
the set of non-empty Dyck paths, where the product will be denoted $\colle$.

This means that
\begin{equation}
  u \colle v = (1, \ov{u} \pscolle \ov{v}, 0).
\end{equation}

The unit of $\monod$ becomes $(1,0)$, and the free generators of $\monod$
are the element $(1,0,1,0)$ and all elements of the shape
$(1,w,1,0,0)$ for some Dyck path $w$.

The monoid structure $\monod$ interacts nicely with the partial order
on $\Dyck_n$.

\begin{lemma}
  \label{cover_colle}
  Let $w \in \Dyck_n$ be a Dyck path. Assume that
  $w = w_1 \colle \dots \colle w_m$ as a product of
  generators of the monoid $\monod$. Let $w \to w'$ be a cover move in
  the partial order on $\Dyck_n$. Then the decomposition of $w'$
  as a product of generators has at most $m$ factors. If this
  decomposition has exactly $m$ factors, then there is one cover move
  inside one of the factors, and the other factors are
  unchanged. Otherwise, several
  consecutive factors are merged into a new factor, all other factors being unchanged.
\end{lemma}
\begin{proof}
  (A) Assume first that the cover move is sliding one full block of $w$,
  and therefore merging two consecutive blocks $v$ and $v'$ of
  $w$. One has to understand the level-decomposition of this new block
  in terms of those of $v=\level(v_1,\dots,v_k)$ and $v'$. This
  depends on the height $1 \leq i \leq k+1$ of the block $v'$ after it
  has been slided.

  Assume first that $1 \leq i \leq k$. Then the new level decomposition
  will be made of $v_1,\dots,v_{i-1}$, then a new term coming from the
  fusion of $v_{i},\dots,v_k$, then the level decomposition of
  $v'$. There are strictly less factors in $w'$ than in
  $w$, because the factor $(1,0,1,0)$ separating the two blocks disappears.

  There remains the case when $i = k + 1$. Then the new level
  decomposition is $v_1,\dots,v_k$, followed by the empty Dyck path, then by the
  level decomposition of $v'$. Therefore the number of factors stays the
  same with just one change, the factor $(1,0,1,0)$ being replaced by the
  factor $(1,1,0,0)$. This corresponds to a cover move inside this factor.

  (B) Assume now that the cover move is happening inside one block $v$ of
  $w$. Let $v = \level(v_1,\dots,v_k)$ be the level-decomposition of
  $v$. Then the cover move can only happen inside some $v_i$ as a cover move
  $v_i \to v'_i$. Therefore in the monoid $\monod$, one gets $w'$ from
  $w$ by replacing the factor $(1,v_i,1,0,0)$ by
  $(1,v'_i,1,0,0)$.
\end{proof}

This means that performing a cover move in a $\colle$ product can
either preserve the product or merge terms of the product, but can never split
them.

Let us now give two lemmas that will be used in the next section to
define a product $\colle$ on the set of all intervals.

\begin{lemma}
  \label{factor_cover_left}
  Let $w \to w'$ be a cover relation of non-empty Dyck paths. Let $u$
  be a non-empty Dyck path. Then there is a cover relation from 
  $u \colle w$ to $u \colle w'$.
\end{lemma}
\begin{proof}
  One can assume that $u$ is not $(1,0)$, because it is the unit of $\monod$.
  Let $\ov{u} = (u', 0^k)$ be the decomposition of $\ov{u}$ according
  to its final sequence of $0$. Let $\ov{w} = w_1 w_2$ and
  $\ov{w'} = w'_1 w'_2$ be the decompositions whose first term is the
  longest prefix before going below the horizontal axis. Then $\ov{u} \pscolle \ov{w}$ and $\ov{u} \pscolle \ov{w'}$ are
  $(u',w_1,0^k,w_2)$ and $(u',w'_1,0^k,w'_2)$. If either $w_1 = w'_1$
  or $w_2 = w'_2$, then the statement is clear, as the cover move is
  inside either $w_1$ or $w_2$. Otherwise, the cover move must merge
  the first two blocks of $w$ by sliding the second block.  One can
  then write $w = (1,w_1,0,x,w_3)$ and $w' = (1, y, 0, w_3)$ where $x$
  is the slided subpath, $y$ is the result of sliding $x$ on $w_1$ and
  $w_3$ may be empty. Then the products $\ov{u}\pscolle\ov{w}$ and $\ov{u}\pscolle\ov{w'}$ are (when
  lifted back to Dyck paths) equal to $(1,u',w_1,0^{k+1},x,w_3)$ and
  $(1,u',y,0^{k+1},w_3)$. One can slide $x$ to get the wanted cover
  move.
\end{proof}

\begin{lemma}
  \label{factor_cover_right}
  Let $w \to w'$ be a cover relation of non-empty Dyck paths. Let $u$
  be a non-empty Dyck path. Then there is a cover relation from 
  $w \colle u$  to $w' \colle u$.
\end{lemma}
\begin{proof}
  One can assume that $u$ is not $(1,0)$, because it is the unit of
  $\monod$. Let $\ov{u} = u' u''$ be the decomposition whose first
  term is the longest prefix before going below the horizontal
  axis. Let $\ov{w} = (v, 0^{k})$ and $w' = (v', 0^{k'})$ be the
  decompositions according to the final sequences of $0$. Then
  $\ov{w}\pscolle\ov{u}$ and $\ov{w'}\pscolle\ov{u}$ are
  $(v, u', 0^k, u'')$ and $(v', u', 0^{k'}, u'')$. If $k = k'$, then
  the statement is clear, as the cover move is inside $v$. Otherwise,
  the cover move must merge the last two blocks of $w$ by sliding the
  last block. One can then write
  $w = (w_1, w_2, 0^{i}, w_3, 0^{k+1})$ and
  $w' = (w_1, w_2, w_3, 0^{i+k+1})$ for some $i \geq 1$, where $w_1$
  may be empty and $(w_3,0^{k+1})$ is the slided subpath. Then the
  right products $\ov{w}\pscolle\ov{u}$ and $\ov{w'}\pscolle\ov{u}$ are (when lifted back to Dyck paths)
  equal to $(w_1, w_2, 0^{i}, w_3, u', 0^{k}, u'', 0)$ and
  $(w_1, w_2, w_3, u', 0^{i+k}, u'', 0)$. Then the expected cover move
  is given by sliding $(w_3, u', 0^{k+1})$, using also the starting
  $0$ of $u''$ or the final $0$ if $u''$ is empty.
\end{proof}

\section{A monoid on intervals}

In this section, we will define and study another monoid $\monoi$, on the set of all
intervals in the posets $\Dyck_n$ for $n \geq 1$.

\smallskip

Let $[w_1,w'_1]$ in $\Dyck_m$ and $[w_2,w'_2]$ in $\Dyck_n$ be two
such intervals.

\begin{lemma}
  In $\Dyck_{m+n-1}$, one has $w_1 \colle w_2 \leq w'_1 \colle w'_2$.
\end{lemma}
\begin{proof}
  Let us pick any sequence of cover moves from $w_1$ to $w'_1$ and any
  sequence of cover moves from $w_2$ to $w'_2$.  Using
  \cref{factor_cover_left} and the chosen cover moves from $w_2$ to
  $w'_2$, one gets a sequence of cover moves from $w_1 \colle w_2$ to
  $w_1 \colle w'_2$. Using then \cref{factor_cover_right} and the
  chosen cover moves from $w_1$ to $w'_1$, one gets a sequence of
  cover moves from $w_1 \colle w'_2$ to $w'_1 \colle w'_2$.
\end{proof}

Let us therefore define a product $\colle$ on the set of all intervals
(except the unique interval in $\Dyck_0$) by the formula
\begin{equation}
  [w_1, w'_1] \colle [w_2, w'_2] = [w_1 \colle w_2, w'_1 \colle w'_2].
\end{equation}
This defines a monoid $\monoi$ on this set of intervals. This monoid
is graded by the size minus $1$. The unit is the unique interval in
$\Dyck_1$.

\begin{proposition}
  The monoid $\monoi$ is free with generators all
  intervals with maximal elements of the shape $(1,w,1,0,0)$ plus
  the interval $[(1,0,1,0),(1,0,1,0)]$ in $\Dyck_2$.
\end{proposition}
Note that the generators of $\monoi$ are exactly the intervals whose
maximal element is a generator of the monoid $\monod$.

\begin{proof}
  Take any interval $[w,w']$. Let $w' = w'_1 \colle \dots \colle w'_k$
  be the unique expression of $w'$ as a product of generators in the
  monoid $\monod$. By \cref{cover_colle}, the minimum $w$ can be
  expressed as $w_1 \colle \dots \colle w_k$ where $w_i \leq w'_i$ for
  every $i$ and the $w_i$ need not be generators in $\monod$. This
  implies that the interval $[w,w']$ can be written as the product
  \begin{equation*}
    [w, w'] = [w_1, w'_1] \colle \dots \colle [w_k, w'_k].
  \end{equation*}
  Therefore the monoid $\monoi$ is indeed generated by the proposed generators.

  Conversely, given a product of some generators, one can recover uniquely the
  generator factors by the same procedure of decomposition of the maximum in
  $\monod$. Therefore the monoid $\monoi$ is free.
\end{proof}

The monoid $\monoi$ interacts nicely with the partial order.
\begin{theorem}
  \label{factoriser}
  The interval $I = [w_1, w'_1] \colle [w_2, w'_2]$ is isomorphic as a poset to the cartesian product of the intervals $[w_1, w'_1]$ and $[w_2, w'_2]$.
\end{theorem}
\begin{proof}
  The isomorphism is just given by the product $\colle$ of Dyck
  paths. By \cref{factor_cover_left} and \cref{factor_cover_right},
  every cover move in the cartesian product of intervals is a cover
  move in $I$. 

  Conversely, by \cref{cover_colle}, every cover move in the interval
  $I$ takes place inside one factor of the maximal element of $I$,
  hence either inside $w'_1$ or $w'_2$.
\end{proof}

The number of generators of the monoid $\monoi$ is a sequence starting with
\begin{equation}
  3,
  3,
  11,
  51,
  267,
  1507,
  8955,
  55251,
  350827,\dots
\end{equation}
which apparently has not already been studied.

\subsection{Specific corollaries}

\label{specific}

Let us now state some interesting special cases of \cref{factoriser}.

First, there is a simple factorisation property, for which we give
another direct proof.

Let $[u,v]$ be an interval in $\Dyck_n$. Let $v_1,\dots,v_k$ be the
unique decomposition of $v$ into blocks. Because the covering moves in
$\Dyck_n$ can only increase the heights, at every point where $v$
touches the horizontal line, the Dyck path $u$ must also touch the
horizontal line. So one can decompose $u$ by cutting at these
points. This defines the Dyck paths $u_i$ for $i=1,\dots,k$.

\begin{proposition}
  \label{facto_tous}
  Every interval $[u,v]$ in $\Dyck_n$ is isomorphic to the cartesian
  product of the intervals $[u_i,v_i]$ for $i = 1,\dots,k$.
\end{proposition}
\begin{proof}
  The same property used above to decompose $u$ is true for all
  elements of the interval $[u,v]$. Cutting at the touch points of $v$
  defines a bijection, with inverse just given by concatenation,
  between elements of $[u,v]$ and elements of the product of the
  intervals $[u_i,v_i]$. This is clearly an isomorphism of posets.
\end{proof}

For $w\in \Dyck_n$, let $I(w)$ be the interval $[\wmin, w]$. As a
special case of \cref{facto_tous}, every interval $I(w)$ is isomorphic
to the product of the intervals $I(w_i)$ over the blocks $w_i$ of $w$.

For $w\in \Dyck_n$, let $J(w)$ be the interval $[\wmin, (1,w,1,0,0)]$
in $\Dyck_{n+2}$.

Keeping the same notations, one has another factorisation result, easily deduced from \cref{factoriser}.
\begin{theorem}
  \label{th_level_iso}
  Let $w$ be a block-indecomposable Dyck path. The interval $I(w)$ is
  isomorphic to the product of the intervals $J(w_i)$, where the $w_i$
  are the Dyck paths in the level-decomposition of $w$.
\end{theorem}

It follows that in this case the isomorphism type of the interval
$I(w)$ only depends on the set of Dyck paths in the
level-decomposition of $w$, and not on their order. Combining with
Prop. \ref{facto_tous}, one obtains the following clean statement.
\begin{theorem}
  \label{th_iso}
  For any $w$, the isomorphism type of the interval $I(w)$ only
  depends on the union of the sets of Dyck paths in the
  level-decomposition of all blocks of $w$.
\end{theorem}

\subsection{Factorisation of principal upper ideals}

\label{princi}

Let us now study the principal upper ideals in $\Dyck_n$. When looking at the
posets $\Dyck_n$ for small $n$, one can see that some of their
principal upper ideals are isomorphic to (products of) principal upper
ideals for some smaller Dyck paths. The main result of this section is
a general description of that phenomenon.

For a Dyck path $w$, let $\up(w)$ be the principal upper ideal
generated by $w$, namely the set of all $u$ such that $w \leq u$.

Let us say that a Dyck path $w$ admits a \defi{strip} if it can be written as
$(u,1,v,1,0,0,0^k)$ where $v$ is any Dyck path. The strip itself is graphically the horizontal region of width $1$ ranging from the letter $1$ after $u$ to the letter $0$ before $0^k$.

\begin{lemma}
  A Dyck path $w$ does not admit a strip if and only if $w$ is empty or
  $w$ ends by $(1,0)$.
\end{lemma}
\begin{proof}
  If $w$ ends by $(1,0)$, it clearly admits no strip.

  Suppose that $w$ does not end with $(1,0)$, and consider the second
  $0$ in the final sequence of $0$. There is exactly one block
  indecomposable Dyck path $x$ inside $w$ that ends with this letter
  $0$. Because the final sequence of $0$ in $x$ has length $2$, $x$
  can be written as expected.
\end{proof}

Suppose now that $w$ admits a strip, so that $w = (u,1,v,1,0,0,0^k)$. Let $u' = (u,1,0,0^k)$.
\begin{proposition}
  \label{princi_factoriser}
  The principal upper ideal
  $\up(w)$ is isomorphic as a poset to the product of $\up(u')$ and $\up(v)$.
\end{proposition}
\begin{proof}
  By the hypotheses on $w$, any sliding move from $w$ either happens
  inside $v$, or somewhere before $v$ in which case it can be matched
  with a sliding move in $u'$. Moreover the resulting $w'$ has the
  same properties as $w$ with modified $u$ or $v$, so that the full
  principal upper ideal factorizes as expected.
\end{proof}
Conversely, one can realize in this way the product of any two
principal upper ideals $\up(u')$ and $\up(v)$, assuming only that the
first one is not empty.

\medskip

As the simplest non-trivial example, $\up((1,0,1,1,0,1,0,1,0,0))$ is isomorphic to the product of $\up((1,0,1,0))$ by itself.

\medskip

Let us define a \defi{reduced interval} as an interval whose minimum
is either empty or of the shape $(w,1,0)$.

From \cref{princi_factoriser}, one can deduce a bijection between
non-reduced intervals and pairs (non-empty interval, interval). A
non-reduced interval is the same as an element in $\up(w)$ for some
non-empty $w$ not of the shape $(w,1,0)$. Therefore one can use the
factorization above to map uniquely this element to a pair of elements
in two arbitrary upper ideals, which give two arbitrary intervals (the first one being non-empty).

At the level of generating series for intervals, one gets from this
decomposition that
\begin{equation}
  f_A = f_R + t (f_A - 1) f_A,
\end{equation}
where $f_A$ is the generating series for all intervals and $f_R$ the
generating series for reduced intervals. With an additional
variable $s$ accounting for the number of blocks in the minimum of the
intervals, one gets the refined equation
\begin{equation}
  \label{eq_fA_fR}
  f_A = f_R + t (f_A - 1) f_A|_{s=1}.
\end{equation}

\section{Properties of the core intervals}

\label{magie}

Let us study now the \textit{core intervals}, namely intervals whose
minimum has the shape $(v,1,0)$ (shape A) and whose maximum has the
shape $(1,w,1,0,0)$ (shape B), where $v$ and $w$ are Dyck paths. In
this section, one assumes that $n \geq 2$.

Let $E_n$ be the subset of elements of $\Dyck_n$ that have either
shape A or shape B.

\begin{lemma}
  The set $E_n$ is a lower ideal in $\Dyck_n$.
\end{lemma}
\begin{proof}
  An element of shape B can only cover elements of shape A
  or B. An element of shape A can only cover elements of shape A.
\end{proof}

It follows that the Hasse diagram of the induced partial order on
$E_n$ is just the restriction of the Hasse diagram of $\Dyck_n$.

For every element $w \in \Dyck_{n-2}$, let us define a chain $E(w)$ of
cover moves in $E_n$. Start from $e_0(w)=(1,0,w,1,0)$. Sliding the
second block of $e_0(w)$ to height $1$ defines $e_1(w)$, which has one
block less then $e_0(w)$. Repeat the same operation and define
successive $e_i(w)$ until reaching a block-indecomposable Dyck path
$e_k(w)$, which has necessarily shape B. The number $k$ is the number
of blocks of $w$ plus $1$.

Suppose that $w$ is a concatenation of blocks $w_1,\dots,w_{k-1}$. One
can describe the chain $E(w)$ completely: the element $e_{i-1}(w)$ for
$1 \leq i \leq k$ is
\begin{equation*}
  (1,w_1,w_2,\dots,w_{i-1},0,w_{i},\dots,w_{k-1},1,0)
\end{equation*}
and the last element $e_k(w)$ is just $(1,w,1,0,0)$. One can therefore recover the
first element of this chain from its last element.

\begin{lemma}
  The set $E_n$ is the disjoint of the chains $E(w)$ for
  $w \in \Dyck_{n-2}$.
\end{lemma}
\begin{proof}
  Consider the following map $\theta$ from $E_n$ to $E_n$. If
  $w \in E_n$ has shape $A$, slide its second block to height $1$ to
  get $\theta(w)$. Otherwise $w$ can be written as $(1,w',1,0,0)$, so
  one can define $\theta(w)$ to be $(1,0,w',1,0)$. By the previous
  remarks, the chains $E(w)$ are nothing but the orbits of $\theta$.
\end{proof}

Note also that the final step in every chain $E(w)$ is a cover move
from shape A to shape B. Every such cover move is the last step of such a chain.

\medskip

The poset $E_n$ is not a disjoint union of total orders, but its structure
is organised around this set of chains as follows.

\begin{proposition}
  \label{autre_chaine}
  Consider a cover relation $u_1 \to u_2$ in $E_n$ where $u_1$ and
  $u_2$ are in two distinct chains $E(w_1)$ and $E(w_2)$. Then there
  exists a cover relation $w_1 \to w_2$ in $\Dyck_{n-2}$.
\end{proposition}
\begin{proof}
  The cover relation $u_1 \to u_2$ must be between two elements of
  shape A or two elements of shape B, because cover moves from shape A
  to shape B are inside chains, as noted above. If both $u_1$ and $u_2$
  have shape B, then the statement is clear. Assume therefore that
  both $u_1$ and $u_2$ have shape A.

  In this situation, $u_1$ can be written as $(1,w'_1,0,w''_1,1,0)$,
  where $w_1 = w'_1 w''_1$ is a concatenation of Dyck paths. The cover
  move $u_1 \to u_2$ can not be the sliding along the $0$ between
  $w'_1$ and $w''_1$, because this move is in the chain $E(w_1)$. It
  can therefore only happen inside $w'_1$ or inside $w''_1$. This
  implies that $u_2$ is $(1,w'_2,0,w''_2,1,0)$, with one part
  unchanged and the other part changed by a cover move. Therefore
  $w_2 = w'_2 w''_2$ and this implies that there is a cover move from
  $w_1$ to $w_2$.
\end{proof}

\begin{proposition}
  Let $v$ be an non-empty Dyck path. The top element of
  the unique chain $E(w)$ containing $(v,1,0)$ is the unique minimal
  element among all Dyck paths $(1,w',1,0,0)$ such that
  $(v,1,0) \leq (1,w',1,0,0)$.
\end{proposition}
\begin{proof}
  Consider any sequence of cover moves from $(v,1,0)$ to some
  $(1,w',1,0,0)$. Note that this takes place entirely inside
  $E_n$. The chosen sequence of cover moves is made either of
  cover moves inside one chain, or of cover moves between chains. By
  \cref{autre_chaine}, the first reached element of shape B
  must be the last element $(1,w'',1,0,0)$ of some chain $E(w'')$ where
  $w \leq w''$. The remaining cover moves are between elements of shape
  B, and therefore $w \leq w'' \leq w'$. The statement follows.
\end{proof}

Let us now use all of this to give a precise description of the set of core intervals. By replacing the bottom
element of such an interval $[u,(1,w',1,0,0)]$ to the top element of its chain $E(w)$, one
gets an interval between elements of shape B, or equivalently an
interval $[w,w']$ in $\Dyck_{n-2}$. The position of $u$ in the chain
$E(w)$ is described by an integer $i$ between $0$ and the number of
blocks of $w$. This defines a map from the set of core intervals
sending $[u,(1,w',1,0,0)]$ to the pair ($[w,w']$, $i$). Conversely, given any
interval $[w,w']$ in $\Dyck_{n-2}$ and any integer $i$ between $0$ and
the number of blocks of $w$, one can recover the full chain $E(w)$ and
pick $u$ by its index in this chain.

To summarize, there is a bijection between core intervals and pairs
made of an arbitrary interval $[w,w']$ and an integer between $0$ and
the number of blocks of $w$.

\section{Counting the intervals}

In this section, we use the previous structural results on intervals to count them.

\begin{theorem}
\label{comptage}
The number of intervals in the poset $\Dyck_n$ is $1$ for $n=0$ and 
\begin{equation}
  3 \frac{2^{n-1} (2n)!}{n!(n+2)!} \quad \text{for}\quad n \geq 1.
\end{equation}
\end{theorem}
This formula describes exactly the sequence \oeis{A000257}, whose
first few terms are
\begin{equation*}
1, 1, 3, 12, 56, 288, 1584, 9152, \dots
\end{equation*}

This is also the number of rooted bicubic planar maps on $2n$ vertices
\cite{tutte}, the number of rooted Eulerian planar maps with $n$ edges, the
number of modern intervals in the Tamari lattice on $\Dyck_n$ and the
number of new intervals in the Tamari lattice on $\Dyck_{n+1}$
\cite{chap_slc}. For a simple bijection between these last two kinds of
intervals, see \cite{rognerud}.

The proof uses a recursive description of all the intervals, based on
the previous structural results. The good catalytic parameter turns
out to be the number of blocks of the minimum of the interval.

Let $A_n$ be the set of all intervals in $\Dyck_n$ for $n \geq 0$. Let
$R_n$ be the subset of $A_n$ made of intervals whose minimum has the
shape $(w,1,0)$ (reduced intervals), plus the interval (empty,
empty). Let $C_n$ be the subset of $A_n$ made of intervals whose
minimum has the shape $(w,1,0)$ and whose maximum has the shape
$(1,w,1,0,0)$ (core intervals), plus the interval
$(1,0,1,0),(1,0,1,0)$.

Let $f_A$, $f_R$, $f_C$ be the associated generating series,
with a variable $t$ for the size and a variable $s$ for the number of
blocks in the minimum of the interval.

From the factorisation property of principal upper ideals in \S \ref{princi}, one gets \eqref{eq_fA_fR}, that we repeat here:
\begin{equation}
  f_A = f_R + t (f_A - 1) f_A|_{s=1}.
\end{equation}

Applying the freeness of the monoid of intervals $\monoi$ to the
subset of intervals whose minimum ends with $(1,0)$ gives
\begin{equation}
  \frac{f_R - 1}{s t} = 1 + \frac{f_A - 1}{st} \frac{f_C}{st},
\end{equation}
because the last factor must be a core interval.

From the properties of core intervals obtained in \S \ref{magie}, one gets
\begin{equation}
  f_C = s^2 t^2 \left( 1 + \frac{s f_A-f_A|_{s=1}}{s - 1} \right).
\end{equation}

All together, these three equations give the functional equation
\begin{equation}
  f_A = 1 + s t + s t (f_A - 1) \left( 1 + \frac{s f_A-f_A|_{s=1}}{s - 1} \right) +  t (f_A - 1) f_A|_{s=1}.  
\end{equation}

Using the general method of \cite{mbmj} to deduce an algebraic
equation from this kind of functional equation with one catalytic
parameter, one gets (as the unique pertinent factor) the equation
\begin{equation}
  16 g^{2} t^{2} - g (8 t^{2} + 12 t - 1) + t^{2} + 11 t - 1
\end{equation}
for the generating series $g = f_A|_{s=1}$, in which one recognizes
the known algebraic equation for the sequence \oeis{A000257}. This
implies \cref{comptage}.

\smallskip

The first few terms of these series are
\begin{equation*}
  f_A = 1 + st + \left(2s^{2} + s\right)t^{2} + \left(5s^{3} + 5s^{2} + 2s\right)t^{3} + \left(14s^{4} + 21s^{3} + 15s^{2} + 6s\right)t^{4} + \cdots
\end{equation*}
\begin{equation*}
  f_R = 1 + st + 2s^{2}t^{2} + \left(5s^{3} + 3s^{2}\right)t^{3} + \left(14s^{4} + 16s^{3} + 8s^{2}\right)t^{4} + \cdots
\end{equation*}
\begin{equation*}
  f_C = 2s^{2}t^{2} + \left(s^{3} + s^{2}\right)t^{3} + \left(2s^{4} + 3s^{3} + 3s^{2}\right)t^{4} + \cdots
\end{equation*}

\subsection{Refinement of enumerative correspondence}

It follows from \cref{comptage} that the number of intervals in
$(\Dyck_n, \leq)$ is the same as the number of modern intervals in the
Tamari partial order on the same set. In this section, a conjectural refinement of
this equality is proposed.

\smallskip





Let us define a map $\kappa$ from Dyck paths of size $n$ to planar
rooted binary trees with $n$ inner vertices, by induction on $n$. When
$n=0$, the empty Dyck path is sent to the binary tree that is just one
leaf. If the Dyck path $w$ is block-indecomposable, it can be written
$w = (1,w',0)$ for some smaller Dyck path $w'$. Then $\kappa(w)$ is
obtained from $\kappa(w')$ by adding one inner vertex (and two leaves) on the rightmost
leaf. Otherwise, cutting $w$ before the last block defines two
smaller Dyck paths $w_1$ and $w_2$ such that $w = w_1 w_2$. Then
$\kappa(w)$ is defined by grafting the root of $\kappa(w_1)$ on the
second leaf from the right of $\kappa(w_2)$.

\begin{proposition}
 The map $\kappa$ is a bijection.
\end{proposition}
\begin{proof}
  It is enough to be able to build the inverse by induction. Let us
  simply sketch the construction. Consider a planar binary tree
  $t$. Let $v$ be the parent vertex of its rightmost leaf.

  If $v$ is directly connected to two leaves, then one can remove $v$
  to get another binary tree $t'$, apply the inverse of $\kappa$ by
  induction to get a Dyck path $w'$ and define $w = (1,w',0)$. Clearly
  $\kappa(w) = t$.

  Otherwise, cut the tree $t$ along the left branch of the vertex
  $v$. This gives two trees $t_1$ (above the cut) and $t_2$ (below the
  cut). Applying the inverse of $\kappa$ by induction gives two Dyck
  paths $w_1$ and $w_2$. Define $w$ to be their concatenation
  $w_1 w_2$. Clearly again $\kappa(w) = t$.
\end{proof}

\begin{proposition}
  The bijection $\kappa$ from Dyck paths to planar rooted binary trees
  is such that, when $t = \kappa(w)$, the number of vertices on
  the rightmost branch of $t$ is the number of final zeros of $w$.
\end{proposition}
\begin{proof}
  By induction on $n$. This is obvious for $n=0$. The
  two possible steps (for block-indecomposable $w$ or otherwise) in
  the inductive definition of the bijection $\kappa$ do preserve this
  property, as can be readily checked.
\end{proof}

\begin{conjecture}
  The bijection $\kappa$ from Dyck paths to planar rooted binary trees
  is such that, when $t = \kappa(w)$, the number of modern intervals
  with minimum $t$ in the Tamari lattice is the size of the dexter
  upper ideal with minimum $w$.
\end{conjecture}


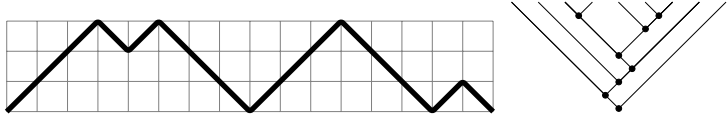
\begin{figure}
\centering
\begin{tikzpicture}[scale=.4]
  \draw[gray,very thin] (0, 0) grid (16, 3);
  \draw[rounded corners=1, color=black, line width=2] (0, 0) -- (1, 1) -- (2, 2) -- (3, 3) -- (4, 2) -- (5, 3) -- (6, 2) -- (7, 1) -- (8, 0) -- (9, 1) -- (10, 2) -- (11, 3) -- (12, 2) -- (13, 1) -- (14, 0) -- (15, 1) -- (16, 0);
\end{tikzpicture}
\includegraphics[width=3cm]{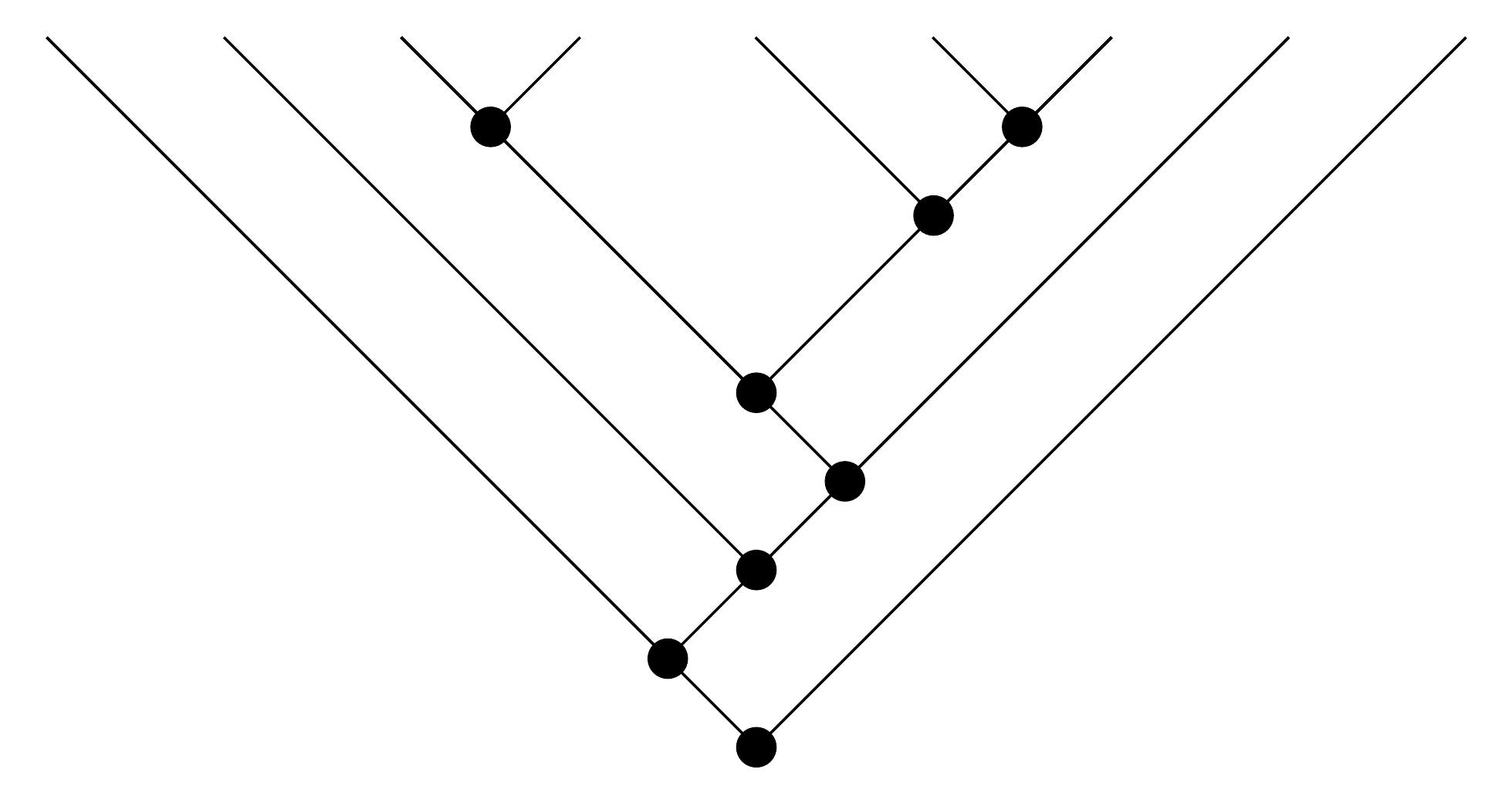}
\caption{Illustration of the bijection $\kappa$}
\end{figure}

\section{Semilattice property}

The aim of this section is to prove that $\Dyck_n$ is a lower
semilattice, namely any two elements have a meet.

\subsection{Some technical lemmas}

This subsection proves several useful lemmas.

\begin{lemma}
  \label{forcer_un}
  Let $u \in \Dyck_n$, with a letter $0$ at position $i+1$. Consider
  the set $R_i(u)$ of elements $v \in \Dyck_n$ such that $u \leq v$, the
  first $i$ letters of $v$ are the first $i$ letters of $u$ and $v$
  has a letter $1$ at position $i+1$.

  Either $R_i(u)$ is empty or $R_i(u)$ has a unique minimal element.
\end{lemma}
This will follow from \cref{sub_forcer_un}.

Keeping the same notations, let us denote by $p$ the prefix of $u$ before
position $i+1$. Then $u$ has a unique expression of the form
\begin{equation}
  \label{frozen_pic}
  u = p \, 0^{\ell} \, X_0 0^{k_0} \, X_1 0^{k_1} \, \dots X_r 0^{k_r} \, X_{r+1},
\end{equation}
where for $i \leq r$ each $X_i$ is a non-empty Dyck path, $\ell >0$,
all $k_i > 0$ and the final $X_{r+1}$ is any Dyck path (possibly
empty). This expression is obtained by starting after the prefix and
repeatedly do the following : go down along the $0$ until reaching a
point before the letter $1$ (or final), then move on the path until reaching the
first point at the same height that is followed by a letter $0$ (or final).

\medskip

Each of the $X_i$ is made of one or several blocks. Such a block is
\textit{movable} if it can be slided without changing the prefix, and
\textit{frozen} otherwise.

In each $X_i$ for $i \leq r$, all the blocks but the last one are
movable. In the last Dyck path $X_{r+1}$, all the blocks are movable.

If there is no movable block in $u$, the set $R_i(u)$ in the statement
of \cref{forcer_un} will be empty. Indeed any cover move not changing
the prefix will keep the shape of the expression \eqref{frozen_pic}
fixed, so that no block becomes movable and $X_0$ can never be slided.

Let us therefore from now on assume the contrary. In this case, let us
denote by $\rise_i(u)$ the element obtained from $u$ by sliding the
first subpath in the first $X_j$ that contains a movable block, up to
the height of the last block of the previous $X_{j-1}$ or up to the end
of the prefix of $u$ if $j = 0$. In the resulting Dyck path, there is
a movable block inside $X_{j-1}$ if $j > 0$.

\begin{lemma}
  \label{sub_forcer_un}
  Let $u \in \Dyck_n$, with a letter $0$ at position $i+1$. Let
  $v \in \Dyck_n$ such that $u \leq v$, the first $i$ letters of $v$
  are the first $i$ letters of $u$ and $v$ has a letter $1$ at
  position $i+1$. Then $\rise_i(u) \leq v$.
\end{lemma}

\begin{proof}
  The proof will be by decreasing induction on $u$.

  Because $u \leq v$, $v$ must be greater than at least one element
  that covers $u$ and does not change the prefix. If
  $\rise_i(u) \leq v$, the statement holds directly.

  The other possible such cover moves can be classified into several types:
  \begin{enumerate}[label=(\arabic*)]
  \item strictly inside one of the blocks of $X_k$ with $k \leq j$,
  \item one block of $X_j$ is slided,
  \item somewhere on the right of $X_j$.

  \end{enumerate}

  Let first assume that there is a cover move $u \leq u'$ of type (1)
  with $u' \leq v$. In $u'$, one finds the same first movable block
  as in $u$. By induction, one also has $\rise_i(u') \leq v$. There
  is a sequence of cover moves $u \leq u' \leq \rise_i(u')$. But these
  cover moves commute, so that $\rise_i(u) \leq \rise_i(u')$.

  Let us now assume that there is a cover move $u \leq u'$ of type (3)
  with $u' \leq v$. In $u'$, one finds the same first movable block
  as in $u$. By induction, one also has $\rise_i(u') \leq v$. There
  is a sequence of cover moves $u \leq u' \leq \rise_i(u')$. But these
  cover moves also commute, so that again
  $\rise_i(u) \leq \rise_i(u')$.

  There remains to assume that there is a cover move $u \leq u'$ of type
  (2) with $u' \leq v$. If the slided block is not the first block or
  the second block of $X_j$, then one can argue as in the preceding
  case, because the first movable block is not affected.

  If the slided block is the second block of $X_j$, then $X_j$ must
  contain at least $3$ blocks, for otherwise the second block is not
  movable.  Then one can check directly that there is a cover
  move $\rise_i(u) \to \rise_i(u')$.

  If the slided block is the first block of $X_j$, it can either be
  slided too high or too low, compared to the exact height of the end
  of $X_{j-1}$ or of the end of the prefix of $u$. In both cases, it
  becomes (part of) a frozen block. The existence of $v$ then forces
  that there is another movable block in $X_k$ for some $k \geq j$.

  If the first block $B$ of $X_j$ is slided too low in $u'$, one can
  apply repeatedly $\rise_i$ to $u'$ until this operator is sliding
  again the slided block $B$ itself. Let us call $u''$ the
  result. Then by induction, one has $u'' \leq v$. One can also see
  that $\rise_i(u) \leq u''$ by repeating on $\rise_i(u)$ all the same
  slidings performed from $u'$ to $u''$.

  The first block $B$ of $X_j$ can only be slided too high in $u'$ if
  there exists $X_{j-1}$, because otherwise the prefix would
  change. In this case, one can apply repeatedly $\rise_i$ to $u'$
  until this operator is making movable the block containing the
  slided block $B$ itself. Let us call $u''$ the result. Then by
  induction, one has $u'' \leq v$. On the other hand, one can apply an
  operator similar to $\rise_i$ but using the second available moving
  block on $\rise_i(u)$ until making the slided block $B$ movable
  again. In other words, one replays on $\rise_i(u)$ all the sliding
  moves performed on $u'$. Let us call the result $\bar{u}$. Then
  there is a cover move $\bar{u} \to u''$.

  In all cases, one obtains that $\rise_i(u) \leq v$.
\end{proof}

Let us now give the proof of \cref{forcer_un}. This is just an
iteration of \cref{sub_forcer_un}, where the pair $(u,v)$ is replaced
by the pair $(\rise_i(u),v)$ that satisfies the same hypotheses as
long as $j > 1$. At each step the index $j$ containing the first
movable block is decreasing by one. In the last step, when $j=1$, the
element $\rise_i(u)$ becomes an element of $R_i(u)$ which is smaller
than any element of $R_i(u)$.

\medskip

For $v,w$ in $\Dyck_n$, let $M(v,w)$ be the set of Dyck paths that are
smaller than both $v$ and $w$. This set is never empty.

\begin{lemma}
  \label{meet_with_prefix}
  Let $v$ and $w$ in $\Dyck_n$ with a common prefix of length
  $i$. Then for every element $u$ in $M(v,w)$, there exists $u'$ in
  $M(v,w)$ such that $u \leq u'$ and $u'$ has the same prefix of
  length $i$ as $v$ and $w$.
\end{lemma}
\begin{proof}
  By induction on $i$. This is true for $i=1$, for $u'=u$. Assume now
  that $v$ and $w$ have a common prefix of length $i+1$ and let $u'_i$
  be defined by induction hypothesis for the shorter common prefix of
  length $i$. If the last letter in the common prefix of length $i+1$
  is the same as the letter of $u'_i$ at position $i+1$, one can take
  $u'$ to be $u'_i$.

  Because $u'_i \in M(v,w)$, there remains only the case where the
  letter of $u'_i$ at position $i+1$ is $0$ and the last letter in the
  common prefix of length $i+1$ of $v$ and $w$ is $1$. Let us apply
  \cref{forcer_un} to $u'_i$. The set $R_{i}(u'_i)$ is not empty as
  it contains $v$ and $w$. One can therefore take $u'$ to be its unique
  minimal element.
\end{proof}

\subsection{Proof of semilattice property}

Let us start with more lemmas.

Let $w$ be a Dyck path such that the $i$-th step in $w$ is $1$ and
this step does not start at height $0$.  On the right from the start
$i_0$ of $i$-th step, move on the path $w$ until meeting a point $i_1$
at the same height and followed by a $0$ step. This must happen, as
the height of $i_0$ is not zero. Between $i_0$ and $i_1$, there is in
$w$ a non-empty sequence of subpaths $x_1, \dots, x_N$.

Let us define another Dyck path $\des_i(w)$ by sliding down in $w$ the
subpath $x_N$ as much as possible, namely by exchanging $x_N$ with all
the consecutive $0$ steps on its right. Then there is a covering move
$\des_i(w) \to w$ that is sliding up the subpath $x_N$.

\begin{lemma}
  \label{sub_force_0}
  Let $w$ be a Dyck path such that the $i$-th step in $w$ is $1$ and
  this step does not start at height $0$. Let $u \leq w$ such that the
  $i$-th letter of $u$ is $0$ and $v$ and $w$ share the same prefix
  before the $i$-th letter. Then $u$ is smaller than $\des_i(w)$.
\end{lemma}
\begin{proof}
  The proof will be by induction on increasing $w$.

  Necessarily, $u$ is smaller than at least one of the elements
  covered by $w$. If $u \leq \des_i(w)$, the statement holds.

  Because of the shared prefix, the other possible down-sliding moves
  from $w$ are of three types:
  \begin{enumerate}[label=(\arabic*)]
  \item strictly inside one of the $x_i$,
  \item splitting some $x_i$,
  \item somewhere on the right of $x_N$.
  \end{enumerate}

  Assume first that $u \leq w'$ where $w' \to w$ is of type (3). By
  induction, $u \leq \des_i(w')$. There is a chain of cover moves
  $\des_i(w') \to w' \to w$. These two cover moves commute if the slided
  subpath in the move $w' \to w$ is not the first subpath after
  $x_N$. Otherwise, one can find a chain of two cover moves from
  $\des_i(w')$ to $\des_i(w)$. In all cases,
  $\des_i(w') \leq \des_i(w)$.

  Assume now that $u \leq w'$ where $w' \to w$ is of type (1). By
  induction, $u \leq \des_i(w')$. There is a chain of cover moves
  $\des_i(w') \to w' \to w$. These cover moves commute, and therefore
  $\des_i(w') \leq \des_i(w)$.

  Assume then that $u \leq w'$ where $w' \to w$ is of type (2). By
  induction, $u \leq \des_i(w')$. There is a chain of cover moves
  $\des_i(w') \to w' \to w$. If $i < N - 1$, these cover moves
  commute, and therefore $\des_i(w') \leq \des_i(w)$.

  If $i=N$, then one can apply twice the induction step to obtain that
  $u \leq \des_i(\des_i(w'))$. But $\des_i(\des_i(w')) \leq \des_i(w)$
  because one can split $x_N$ after it has been slided down.

  If $i = N-1$, one can also apply twice the induction step, to get
  that $u \leq \des_i(\des_i(w'))$. One then checks that
  $\des_i(\des_i(w')) \leq \des_i(w)$ holds also in this case, by just
  one cover move.

  One therefore deduces the statement in all cases.
\end{proof}

Keeping the same notations, let $s_i(w)$ denote $\des_i^N(w)$, the
image of $w$ under the $N$-times iteration of the application
$\des_i$.

\begin{lemma}
  \label{force_0}
  Let $w$ be a Dyck path such that the $i$-th step in $w$ is $1$ and
  this step does not start at height $0$. Let $S_i(w)$ be the set
  of all Dyck paths $u$ such that $u \leq w$, the $i$-th letter of $u$
  is $0$ and $u$ and $w$ share the same prefix before the $i$-th letter.
  Then $s_i(w)$ is the unique maximal element of $S_i(w)$.
\end{lemma}
\begin{proof}
  First note that $s_i(w)$ is indeed in $S_i(w)$.

  Let $u$ be an element of $S_i(w)$. One can apply \cref{sub_force_0}
  to the pair $(u,w)$ to get another pair $(u, w')$ where
  $w'=\des_i(w)$. Either $w' = s_i(w)$, or this new pair satisfies
  again the hypotheses of \cref{sub_force_0}. One can therefore repeat
  this exactly $N$ times, until reaching the pair $(u,s_i(w))$. In
  particular, $u \leq s_i(w)$.
\end{proof}

\begin{theorem}
  The poset $(\Dyck_n, \sm)$ is a meet-semilattice.
\end{theorem}
\begin{proof}
  Let $v$ and $w$ be two Dyck paths, and let us look for their
  meet. One can assume that $v$ and $w$ are not equal.

  Start from the left, until meeting a difference between $v$ and
  $w$. Let $i$ be the last common point. One can assume that $w$ is
  above $v$ just after the point $i$, by exchanging $v$ and $w$ if
  necessary. Note that the height of $i$ cannot be zero.

  One can therefore apply \cref{force_0} to $w$ for the step after
  position $i$, and obtain an element $w'$ which is maximal among all
  elements smaller than $w$ that share the same prefix followed by
  the letter $0$.

  This gives a new pair of elements $(v,w')$ with $w' \leq w$. Let us
  prove that $M(v,w) = M(v,w')$. The inclusion $M(v,w') \subseteq
  M(v,w)$ is clear because $w' \leq w$.

  Conversely, let $u$ be an element of $M(v,w)$. Using
  \cref{meet_with_prefix}, one can find $u'$ in $M(v,w)$ with $u \leq
  u'$ and $u'$ share the common prefix of $v$ and $w$. Note that the
  first letter after the common prefix in $u'$ is $0$, because $u'
  \leq v$. It therefore follows from the definition of $w'$ that $u'
  \leq w'$.

  Hence $M(v,w) = M(v,w')$, and the common prefix of $v$ and $w'$ is
  strictly longer. Therefore iterating this whole procedure on pairs
  of Dyck paths ends at a common Dyck path $z$, that is smaller than
  $v$ and $w$ and such that $M(z,z) = M(v,w)$. This $z$ is therefore
  the meet of $v$ and $w$.
\end{proof}

\section{Derived equivalences of intervals}

We now turn briefly to a more subtle equivalence between intervals,
namely \defi{derived equivalence}, which is defined as follows. One
can consider any finite poset $P$ as a small category, with a unique
morphism $x \to y$ if and only if $x \leq y$. Then the category of
modules over $P$ with coefficients in some base field $K$ can be
defined as the category of functors from $P$ to finite-dimensional
vector spaces over $K$. This is an abelian category, with enough
projectives and injectives, and finite global dimension. One can
therefore associate to $P$ the (bounded) derived category
$\deri_{K}(P)$ of this category of modules.

Two posets $P$ and $Q$ are said to be \defi{derived equivalent} (over
$K$) if there is a triangle-equivalence between $\deri_K(P)$ and
$\deri_K(Q)$. 


In this section, we conjecture the existence of derived equivalences
between some particular kinds of intervals in $\Dyck_n$.

Recall from section \ref{specific} that the intervals $J(w) = I((1,w,1,0,0))$ are the factors in
the cartesian factorisation of the intervals $I(w)$.

\begin{conjecture}
  \label{th_iso_derive}
  For any $w$, the derived isomorphism type of the interval $J(w)$ only
  depends on the union of the sets of Dyck paths in the
  level-decomposition of all blocks of $w$.
\end{conjecture}

This conjecture is based on experimental evidence, namely the
coincidence of some invariants of posets (Coxeter polynomials) which
only depend on the derived categories.

Note the striking similarity of this conjecture with theorem
\ref{th_iso}. This conjecture may even be a characterisation of
derived equivalence classes.

\medskip

As a special case, if two words $w$ and $w'$ are related by a
permutation of their blocks, then $J(w)$ and $J(w')$ should be
derived-equivalent. As the simplest possible non-trivial example, let
us consider the posets $J(1,0,1,1,0,0)$ and $J(1,1,0,0,1,0)$. Both
have $9$ elements and share the same Coxeter polynomial
$\Phi_1^2 \Phi_2 \Phi_3 \Phi_5$, where the $\Phi_d$ are the cyclotomic
polynomials. These posets are in fact related by a flip-flop in the
sense of Ladkani \cite{flipflop} (mapping two elements near the top of
$J(1,1,0,0,1,0)$ to two elements near the bottom of $J(1,0,1,1,0,0)$),
and therefore derived-equivalent.

As another special case, if two block-indecomposable words $w$ and
$w'$ are related by a permutation of their level-decomposition, then
$J(w)$ and $J(w')$ should be derived-equivalent. As the simplest
possible non-trivial example, let us consider the posets
$J(1,1,0,1,1,0,0,0)$ and $J(1,1,1,0,1,0,0,0)$. Both have $27$ elements
and share the same Coxeter polynomial $\Phi_2 \Phi_4 \Phi_{18}
\Phi_{54}$. It is not clear if these intervals are derived equivalent.

To illustrate the general case in the simplest possible way, consider
the posets $J(1,0,1,1,1,0,0,0)$ and $J(1,1,0,0,1,1,0,0)$. Both have
$20$ elements and share the same Coxeter polynomial
$\Phi_1^2 \Phi_2^2 \Phi_{3} \Phi_{5} \Phi_6^2 \Phi_7$. In this case,
there is no obvious flip-flop to prove the expected derived
equivalence.

\subsection{About the notion of $f$-vector}

Beware that this section is very speculative.

The classical notion of $f$-vector is attached to cellular or
simplicial complexes, where it records the number of cells of every
dimension.

In some famillies of posets, including the one studied here, but also
the Tamari lattices, the cambrian lattices and many of their relatives, the pictures of the
Hasse diagram of intervals, when visually inspected by the human eye,
strongly suggest the existence of a cellular complex whose skeleton
would be the Hasse diagram.

Although we will not try to give and justify a precise definition
here, there is one way to find the $f$-vector of this putative
cell-complex, using only the partial order. Namely, every cell can be
identified with its minimal and maximal elements. These two elements
must form something like a minimal spherical interval in the poset.

There is a general phenomenon, observed in several families of posets,
that derived equivalence of posets often come together with an equality
between $f$-vectors. One important instance is the conjectured derived
equivalence between cambrian lattices and lattices of order ideals in
root posets.

This phenomenon seems also to be present in the intervals of the dexter
lattices, at least in the intervals $J(w)$. One expects that derived
equivalent $J(w)$ will share the same $f$-vector, but not conversely.

\section{The Hochschild polytope as an interval}

In this section, we explain an unexpected connection between a
specific interval in $\Dyck_{n+2}$ and a cell complex called the
Hochschild polytope, introduced in algebraic topology by Saneblidze
\cite{san_arxiv, san_note}. Our initial reason for looking at this
particular interval was its appearance inside the interval of largest
cardinality among all $I(w)$.

For $n \geq 1$, let $F_n$ be the interval in $\Dyck_{n+2}$ between
$(1,1,0,0,(1,0)^{n})$ and $(1,1^n,0^n,1,0,0)$. This is indeed an
interval, as one can move from the former to the latter by sliding (in
their order from left to right) all the initial blocks $(1,0)$
(all of them are slided to the maximal possible height, except the
last one that is slided to height $1$).

\begin{figure}[h]
\centering
\begin{tikzpicture}[>=latex,line join=bevel,scale=0.25]
\node (node_1) at (89.0bp,135.0bp) [draw,draw=none] {$\vcenter{\hbox{$\begin{tikzpicture}[scale=.15]  \draw[gray,very thin] (0, 0) grid (6, 2);  \draw[rounded corners=1, color=black, line width=2] (0, 0) -- (1, 1) -- (2, 2) -- (3, 1) -- (4, 2) -- (5, 1) -- (6, 0);\end{tikzpicture}$}}$};
  \node (node_0) at (89.0bp,33.0bp) [draw,draw=none] {$\vcenter{\hbox{$\begin{tikzpicture}[scale=.15]  \draw[gray,very thin] (0, 0) grid (6, 2);  \draw[rounded corners=1, color=black, line width=2] (0, 0) -- (1, 1) -- (2, 2) -- (3, 1) -- (4, 0) -- (5, 1) -- (6, 0);\end{tikzpicture}$}}$};
  \draw [blue,->,dashed] (node_0) ..controls (89.0bp,74.031bp) and (89.0bp,83.221bp)  .. (node_1);
\end{tikzpicture}
\begin{tikzpicture}[>=latex,line join=bevel,scale=0.25]
\node (node_4) at (243.0bp,381.0bp) [draw,draw=none] {$\vcenter{\hbox{$\begin{tikzpicture}[scale=0.15]  \draw[gray,very thin] (0, 0) grid (8, 3);  \draw[rounded corners=1, color=black, line width=2] (0, 0) -- (1, 1) -- (2, 2) -- (3, 3) -- (4, 2) -- (5, 1) -- (6, 2) -- (7, 1) -- (8, 0);\end{tikzpicture}$}}$};
  \node (node_3) at (370.0bp,251.0bp) [draw,draw=none] {$\vcenter{\hbox{$\begin{tikzpicture}[scale=0.15]  \draw[gray,very thin] (0, 0) grid (8, 3);  \draw[rounded corners=1, color=black, line width=2] (0, 0) -- (1, 1) -- (2, 2) -- (3, 3) -- (4, 2) -- (5, 1) -- (6, 0) -- (7, 1) -- (8, 0);\end{tikzpicture}$}}$};
  \node (node_2) at (117.0bp,251.0bp) [draw,draw=none] {$\vcenter{\hbox{$\begin{tikzpicture}[scale=0.15]  \draw[gray,very thin] (0, 0) grid (8, 2);  \draw[rounded corners=1, color=black, line width=2] (0, 0) -- (1, 1) -- (2, 2) -- (3, 1) -- (4, 2) -- (5, 1) -- (6, 2) -- (7, 1) -- (8, 0);\end{tikzpicture}$}}$};
  \node (node_1) at (144.0bp,135.0bp) [draw,draw=none] {$\vcenter{\hbox{$\begin{tikzpicture}[scale=0.15]  \draw[gray,very thin] (0, 0) grid (8, 2);  \draw[rounded corners=1, color=black, line width=2] (0, 0) -- (1, 1) -- (2, 2) -- (3, 1) -- (4, 2) -- (5, 1) -- (6, 0) -- (7, 1) -- (8, 0);\end{tikzpicture}$}}$};
  \node (node_0) at (217.0bp,33.0bp) [draw,draw=none] {$\vcenter{\hbox{$\begin{tikzpicture}[scale=0.15]  \draw[gray,very thin] (0, 0) grid (8, 2);  \draw[rounded corners=1, color=black, line width=2] (0, 0) -- (1, 1) -- (2, 2) -- (3, 1) -- (4, 0) -- (5, 1) -- (6, 0) -- (7, 1) -- (8, 0);\end{tikzpicture}$}}$};
  \draw [red,->] (node_0) ..controls (251.39bp,77.124bp) and (261.29bp,90.01bp)  .. (270.0bp,102.0bp) .. controls (292.05bp,132.36bp) and (315.59bp,167.08bp)  .. (node_3);
  \draw [blue,->,dashed] (node_0) ..controls (187.31bp,74.669bp) and (180.04bp,84.625bp)  .. (node_1);
  \draw [blue,->,dashed] (node_3) ..controls (314.98bp,307.46bp) and (305.27bp,317.24bp)  .. (node_4);
  \draw [blue,->,dashed] (node_1) ..controls (133.49bp,180.36bp) and (130.02bp,195.02bp)  .. (node_2);
  \draw [red,->] (node_2) ..controls (160.81bp,296.5bp) and (175.79bp,311.73bp)  .. (node_4);
\end{tikzpicture}
\caption{The Hochschild intervals $F_1$ and $F_2$.}
\end{figure}
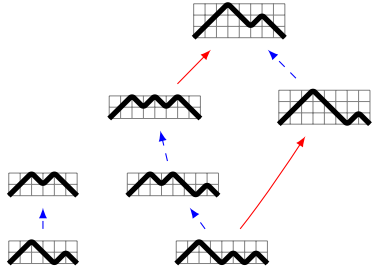

Let us define a \textit{valley} in a Dyck path to be a subword
$(0,1)$, which means a local minimum for the height
function. Similarly, a \textit{peak} is a local maximum of height.

Note that all elements of $F_n$ start with $(1,1)$. Moreover, all the
valleys in all elements of $F_n$ have height $0$ or $1$. This follows
from the next lemma, as this is true for the maximum of $F_n$.

\begin{lemma}
  \label{hauteur01}
  Let $w$ be a Dyck path having only valleys at height $0$ or
  $1$. If $y \leq w$, then $y$ has the same property.
\end{lemma}
\begin{proof}
  It is enough to prove this when $y$ is covered by $w$.

  If $w$ has a valley at height $0$, then $y$ has the same valley. One
  can therefore assume that $w$ is block-indecomposable, and has only
  valleys at height $1$. Then one can check that all possible down
  cover moves from $w$ can only create a valley at level $0$ or $1$.
\end{proof}

\begin{lemma}
  \label{ht_decroit}
  Let $w \in F_n$. Then the height of the valleys in $w$ is weakly
  decreasing from left to right.
\end{lemma}
\begin{proof}
  Otherwise, there is a valley of height $0$ followed by a valley of
  height $1$. The subpath after the valley of height $0$ must be
  slided at some point in any chain of cover moves from $w$ to the maximum
  of $F_n$, but this would create a valley of height at least
  $2$. This is absurd.
\end{proof}

\begin{lemma}
  \label{picfin}
  For $n \geq 1$, every element of $F_n$ ends either with $(0,1,0)$ or
  with $(0,1,0,0)$.
\end{lemma}
\begin{proof}
  The only other possibility is to end with $(1,1,0,0)$,
  because of the shape of the maximum of $F_n$. But then this final
  subpath can not be slided with the result being still below
  the maximum of $F_n$. This is absurd.
\end{proof}

\begin{lemma}
  \label{intrinsics}
  The set $F_n$ can be described as the set of Dyck paths starting
  with $(1,1)$, having only valleys of height $0$ or $1$, where these
  heights are decreasing from left to right, and ending either by
  $(0,1,0)$ or by $(0,1,0,0)$.
\end{lemma}
\begin{proof}
  Let us call $Q_n$ this set of elements. By the preceding lemmas and remarks,
  $Q_n$ contains $F_n$. For the converse inclusion, one only needs to
  check the two following statements: (1) an element of $Q_n$ which is
  not the maximum of $F_n$ can be covered by another element of $Q_n$
  and (2) an element of $Q_n$ which is not the minimum of $F_n$ covers
  another element of $Q_n$.

  (1) Let $z$ be an element of $Q_n$. If $z$ has exactly two peaks,
  then $z$ is either the maximum of $F_n$ or is covered by this
  maximum. One can therefore assume that there are at least three
  peaks in $z$, hence at least two valleys. If there is a valley of height
  $0$, one can slide up by one step the subpath after the first valley
  of height $0$, which becomes a valley of height $1$. The result is
  still in $Q_n$. If all valleys have height $1$, one can merge
  the first two peaks.

  (2) Let $z$ be an element of $Q_n$. If it ends with $(0,1,0,0)$, one
  can slide down this last peak and get another element of
  $Q_n$. Otherwise, $z$ ends with $(0,1,0)$. If it has a valley at
  height $1$, one can slide down the subpath after the rightmost
  valley of height $1$. Otherwise one can take any peak of height at
  least $2$ and cut it into $2$ smaller peaks, except when the only
  available peak of height at least $2$ is at the beginning of
  $z$. But then $z$ is the minimum of $F_n$.
\end{proof}

\begin{remark}
  \label{two_peaks}
  Note that being block-indecomposable is equivalent inside $F_n$ to
  having no valley at height $0$. Note also that for $n \geq 1$ every
  element of $F_n$ has at least two peaks, because the unique path
  with just one peak is not in $F_n$.
\end{remark}

Let us define $F_{n,b}$ as the subset of block-indecomposable elements
of $F_{n}$.

\begin{lemma}
  \label{booleen_top}
  For $n \geq 1$, the subset $F_{n,b}$ form a boolean lattice of
  cardinality $2^{n-1}$ with minimum $w_{n,b} = (1,1,(0,1)^n,0,0)$ and the
  same maximum as $F_n$.
\end{lemma}
\begin{proof}
  Because all valleys must have height $1$, the shape of possible Dyck
  paths is strongly constrained by \cref{intrinsics}. The rightmost
  peak must have height $2$, and this forces at least one valley just
  before the final $(1,0,0)$. Every subset of the valleys of $w_{n,b}$
  that contains this rightmost one defines a unique element in the
  interval. The induced partial order is given by inclusion of subsets
  of valleys.
\end{proof}

\begin{remark}
  \label{split_Fb}
  This boolean lattice on $F_{n,b}$ can be written as the disjoint union of
  two boolean lattices of half cardinality: elements ending with
  $(1,0,1,0,0)$ (bottom part) and the others (top part).
\end{remark}

\begin{lemma}
  \label{booleen_bot}
  The subset of $F_n$ of elements having only valleys of height $0$ is
  a boolean lattice of cardinality $2^{n-1}$.
\end{lemma}
\begin{proof}
  The proof is very similar to the proof of \cref{booleen_top}.
\end{proof}

Let us now define a map $\rho$ from $F_n$ to some set of words of
length $n$ in the alphabet $\{0,1,2\}$. Let $w$ be a Dyck path in
$F_n$. One reads the word $w \in F_n$ from left to right, while
keeping track of an integer $N_2$ (initially set to $0$) and some
prefix of the image $\rho(w)$ (initially the empty word). When two
consecutive $1$ are read in $w$ (except the first two letters of $w$),
the integer $N_2$ is increased by $1$. When a valley of $w$ is read
(with height $h$ being either $0$ or $1$ by \cref{hauteur01}), the
word $[h,2^{N_2}]$ (where the power means that the letter $2$ is repeated) is appended to the current prefix, and $N_2$ is
then set back to $0$. The result $\rho(w)$ is the prefix obtained
after reading all of $w$. The length of $\rho(w)$ is $n$ because every
letter $1$ in $w$ (except the two initial ones) contributes a letter in $\rho(w)$.

For example, the image by $\rho$ of $(1,1,1,0,0,1,0,0,1,0) \in F_3$ is
$[1,2,0]$. The minimal element of $F_n$ is mapped by $\rho$ to the
word $[0,\dots,0]$, and the maximal element to the word
$[1,2,\dots,2]$. The image of the Dyck path $w_{n,b}=(1,1,(0,1)^n,0,0)$
is $[1,\dots,1]$.

By the map $w \mapsto \rho(w)$, the number of valleys at height $0$
(resp $1$) of $w$ becomes the number of letters $0$ (resp. $1$) in
$\rho(w)$.

The construction of $\rho$ can be reversed as follows, proving that it
is injective. Starting from any element $z$ in the image $\rho(F_n)$:

(1) split $z$ as a sequence of bricks $[1,2,\dots,2]$ and
$[0,2,\dots,2]$, formed by a letter $0$ or $1$ and the maximal
sequence of following $2$.

(2) start a new Dyck path at height $0$. For each brick, use the number $N_2$ of $2$ in this
brick to move up to the appropriate height (one more for the initial
brick), then move down to height $0$ or $1$ according to the first
letter of the brick.

(3) when the list of bricks is exhausted, move up by one step and go
down to height $0$.

For example, one can find in this way that the pre-image of $[1,0,2]$ by $\rho$, which has
two bricks $[1]$ and $[0,2]$, is $(1,1,0,1,1,0,0,0,1,0)$.

\begin{lemma}
  \label{cover_increase}
  For a cover move $w \to w'$ in $F_n$, the words $\rho(w)$ and
  $\rho(w')$ differ by exactly one letter, which increases.
\end{lemma}
\begin{proof}
  One has to distinguish two kinds of cover moves. If the number of
  peaks is unchanged, then one valley of height $0$ becomes a valley
  of height $1$. On the image by $\rho$, only the letter encoding this
  height is modified. Otherwise, the number of peaks is decreased by
  one and the two associated bricks in the image by $\rho$ become just
  one brick. The valleys corresponding to these two bricks have the
  same height. In the image by $\rho$, this means that one subword
  $[x,2^M,x,2^N]$ is replaced by $[x,2^{M+1+N}]$, where $x$ is either
  $0$ or $1$.
\end{proof}

\medskip

Let us define $F_{n,0}$ and $F_{n,1}$ as the partition of $F_n$
according to the height of the first valley. This corresponds to the
decomposition of $\rho(F_n)$ according to the first letter.

\begin{lemma}
  \label{intervalle_Fn1}
  For $n \geq 1$, the set $F_{n,1}$ is the interval between
  the element $w_{n,1} = (1,1,0,1,0,0,(1,0)^{n-1})$ and the maximum of $F_n$.
\end{lemma}
\begin{proof}
  The property ($\clubsuit$) of having the first valley at height $1$ is
  preserved by cover moves inside $F_n$, that can only delete a valley
  of height $1$ if it is followed by another valley of height $1$.

  Therefore all the elements that are greater than $w_{n,1}$ have property ($\clubsuit$).

  Conversely, for any element $w \not= w_{n,1}$ with property ($\clubsuit$),
  one can find a smaller element $w'$ with property ($\clubsuit$). Either
  one can easily find such an element $w'$ having the shape
  $(1,(1,0)^k,0,(1,0)^\ell)$ (by lowering the highest peaks) or $w$ itself
  has this shape. In the latter case, one can take $w'= w_{n,1}$.

  It follows that $w_{n,1}$ is the unique minimal element with property ($\clubsuit$).
\end{proof}

Let us also denote
$F_{n,1,0}$, $F_{n,1,1}$ and $F_{n,1,2}$ for the partition of
$F_{n,1}$ according to the last letter of the image by $\rho$. 

From \cref{cover_increase} and \cref{intervalle_Fn1}, one deduces:
\begin{lemma}
  The set $F_{n,1,2}$ is an upper ideal in $F_n$.
\end{lemma}

\begin{lemma}
  \label{F1_F12}
  For $n \geq 1$, the map that inserts $(1,0)$ at the top of the
  next-to-rightmost peak defines a bijection from $F_{n,1}$ to
  $F_{n+1,1,2}$. This corresponds to adding $2$ at the end of
  $\rho(w)$.
\end{lemma}
\begin{proof}
  This is easily checked directly for $n=1$. Let us assume that $n \geq 2$.

  First, one deduces from \cref{intrinsics} that adding $(1,0)$ at the
  top of the next-to-rightmost peak defines a map from $F_{n}$ to
  $F_{n+1}$. This clearly preserves the height of the first valley,
  hence defines an injective map from $F_{n,1}$ to $F_{n+1,1}$.

  Using the definition of $\rho$, this application does add $2$ at the
  end of the image by $\rho$, because it increases the height of the
  next-to-last peak by one.  Therefore its image is contained in
  $F_{n+1,1,2}$.



  Conversely for any element $z$ of $F_{n+1,1,2}$, one can remove
  $(1,0)$ on the top of the next-to-rightmost peak to define a Dyck
  path $x$. This does not change the height of the valleys, in
  particular the first valley of $x$ has height $1$. One just needs to
  prove that $x$ is in $F_n$. Using \cref{intrinsics}, one needs only
  to check the conditions at the beginning and at the end. The
  condition that $x$ starts by $(1,1)$ can fail if and only if the
  next-to-rightmost peak of $z$ is its first peak and has height
  $2$. This only happen when $z$ is the maximal element of $F_1$, but
  $z$ belongs to $F_{n+1}$ for some $n \geq 1$. The condition at the
  end is ensured by the final letter $2$ in $\rho(z)$, which implies
  that removing $(1,0)$ on its top does not delete the
  next-to-rightmost peak.
\end{proof}

Let us define a map $\mu$ on $F_n$ by adding $(1,0)$ at the end.
\begin{lemma}
  The image of $\mu$ is contained in $F_{n+1}$.
\end{lemma}
\begin{proof}
  This follows easily from \cref{intrinsics}.
\end{proof}

Using the definition of $\rho$, one can check that adding $(1,0)$ at
the end of any $w$ in $F_n$ corresponds to adding $0$ at the end of
$\rho(w)$.

\begin{lemma}
  \label{F1_F10}
  For $n \geq 1$, the map $\mu$ is a bijection from $F_{n,1}$ to
  $F_{n+1,1,0}$.
\end{lemma}
\begin{proof}
  Applying $\mu$ on an element of $F_{n,1}$ gives an element
  of $F_{n+1,1,0}$. This is clearly an injective map.

  Conversely, any element $w$ of $F_{n,1,0}$ must end with $(1,0)$,
  because its last valley has height $0$. 

  Using \cref{intrinsics}, one can show that cutting this final
  $(1,0)$ gives an element of $F_{n,1}$ whose image by $\mu$ is
  $w$. The only required check is the condition at the end, which
  follows from the hypothesis that the last letter of $\rho(w)$ is
  $0$. This proves the surjectivity.
\end{proof}

\begin{lemma}
  \label{Fb_F11}
  The map that inserts $(1,0)$ just before the final letter $0$ defines a
  bijection from $F_{n,b}$ to $F_{n+1,1,1}$. This corresponds to
  adding $1$ at the end of $\rho(w)$.
\end{lemma}
\begin{proof}
  By \cref{booleen_top} and the remark following it, there are two
  inclusions of the set $F_{n,b}$ in $F_{n+1,b}$. One can check that
  the bottom inclusion is given by inserting $(1,0)$ just before the
  final letter $0$. Through the application of $\rho$, this amounts to
  adding a final $1$ to $\rho(w)$. The image is therefore in
  $F_{n+1,1,1}$. This is clearly an injective map.

  Conversely, let $w$ in $F_{n+1,1,1}$. Because the last letter in
  $\rho(w)$ is $1$, and using \cref{ht_decroit}, all valleys of $w$
  have height $1$. Moreover the Dyck path $w$ must end with $(1,0,0)$,
  because it is smaller than the maximum of $F_n$. But it must in fact
  end with $(0,1,0,1,0,0)$, for otherwise the final letter of $\rho(w)$
  would be $2$. Hence one can remove $(1,0)$ just before the final
  $0$, and get an element of $F_{n,b}$, whose image is $w$.
\end{proof}

\begin{lemma}
  \label{Fx_F0}
  The map that slides down the subpath after the first valley defines
  a bijection from the subset of $F_{n,1}$ where only the first valley
  has height $1$ to $F_{n,0}$. It amounts to replacing the first
  letter of $\rho(w)$ by $0$.
\end{lemma}
\begin{proof}
  Let us consider an element $w$ of $F_{n,1}$ with one valley of
  height $1$ and all other valleys have height $0$. There is a unique
  element $w'$ in $F_{n+1,0}$ that is covered by $w$, which is obtained by
  sliding down the subpath after the first valley in $w$. One can
  check that $\rho(w')$ is obtained by replacing the first letter of
  $\rho(w)$ by $0$. 

  Conversely, let $w'$ in $F_{n,0 }$. Then $w'$ has only valleys of
  height zero. Let $w$ be obtained by sliding up the subpath after the
  first valley of $w'$, by just one step in order to create a valley
  of height $1$. Then one can check that $w'$ is still in $F_n$ using
  \cref{intrinsics}, because the height of the rightmost peak is
  always at most $2$.

  Moreover $w$ has exactly one valley of height $1$. These two constructions
  are clearly inverses of each other.
\end{proof}

\medskip

The Hasse diagram of the interval $F_n$ looks like the graph of
vertices and edges of some polytope. It turns out to be related to a
familly of cell complexes due to Saneblidze. Namely, one can identify
its image by $\rho$ as the set of vertices defined by
Saneblidze in \cite{san_arxiv}.

\begin{theorem}
  \label{th_san}
  The interval $F_n$ is mapped by $\rho$ to the set of coordinates of
  the Hochschild polytope of Saneblidze.
\end{theorem}

\medskip

Before entering the proof, let us start by giving a recursive
description of these sets $Z_n$ inside $\{0,1,2\}^n$, extracted
carefully from this reference and reformulated in simpler terms, as
the disjoint union of subsets $Z_n = Z_{n,0} \sqcup Z_{n,1}$. The
recursive description also involves a subset
$Z_{n,b} \subseteq Z_{n,1}$.

For $n=1$, this is given by $Z_{n,0} = \{ [0] \}$ and $Z_{n,1} = Z_{n,b} = \{ [1] \}$.

For $n \geq 1$, the description is given by:
\begin{enumerate}[label=(\roman*)]
\item $Z_{n+1,1}$ is made of all elements $[z,0]$ and $[z,2]$ for
  $z \in Z_{n,1}$ and all elements $[z,1]$ for $z \in Z_{n,b}$,
\item $Z_{n+1,b}$ is made of all elements
  $[z,1]$ and $[z,2]$ for $z \in Z_{n,b}$.
\item The subset $Z_{n+1,0}$ is made by replacing the initial letter by
$0$ in all elements of $Z_{n+1,1}$ in which the letter $1$ only appears
as the first letter.
\end{enumerate}

By induction, all elements of $Z_{n,0}$ (resp. $Z_{n,1}$ and
$Z_{n,b}$) start with $0$ (resp. $1$). Note also that the elements of
$Z_{n,b}$ only contains the letters $1$ and $2$.

\begin{proposition}
  The image of $F_n$ by $\rho$ is equal to $Z_n$. Moreover $F_{n,b}$
  is mapped to $Z_{n,b}$, $F_{n,0}$ to $Z_{n,0}$ and $F_{n,1}$ to
  $Z_{n,1}$.
\end{proposition}
\begin{proof}
  By induction on $n \geq 1$. The statement holds by inspection for $n=1$.

  Let us assume that the statement holds up to $n - 1$. We then need to
  perform a decomposition of $F_{n}$ that is parallel to the
  recursive definition of $Z_{n}$.

  First, $F_{n}$ is the disjoint union of $F_{n,0}$ and
  $F_{n,1}$. Similarly, $Z_{n}$ is the disjoint union of $Z_{n,0}$
  and $Z_{n,1}$, according to the first letter. It is therefore
  enough to work separately on each part, starting with $F_{n,1}$.

  Using the decomposition of $F_{n,1}$ into three parts according to
  the last letter of the image by $\rho$, and the bijections stated in
  \cref{F1_F10}, \cref{F1_F12} and \cref{Fb_F11}, one obtains that
  $\rho(F_{n,1})$ has the defining property (i) of $Z_{n,1}$.

  In order to check that $\rho(F_{n,b})$ has the defining property
  (ii) of $Z_{n,b}$, one can use \cref{booleen_top} and the remark
  following it to check explicitly that elements of $\rho(F_{n+1,b})$ are obtained
  by adding either $1$ or $2$ at the end of elements of
  $\rho(F_{n,b})$.
  
  There remains to check that $\rho(F_{n,0})$ has the defining
  property (iii) of $Z_{n,0}$. This is exactly provided by \cref{Fx_F0}.
\end{proof}

\medskip

Words on the alphabet $\{0,1,2\}$ in $Z_n$ are considered as vectors
in $\ZZ^n$. The partial order on $Z_n$ introduced by Saneblidze is
termwise-comparison: two words $z$ and $z'$ are comparable if and only
if $z_i \leq z'_i$ for all $1 \leq i \leq n$.

One certainly expects that the map $\rho$ should give an isomorphism
from the partial order on $F_n$ to this partial order on $Z_n$. One
direction of the proof is just \cref{cover_increase}. Completing the
proof would require a precise study of the covering relations in
both partial orders. Because this does not seem to be central enough
in the present article, this is left for another work.











As a corollary of what precedes, one can count the elements of
$F_n$ as follows. From \cref{booleen_bot}, \cref{booleen_top},
\cref{F1_F10}, \cref{F1_F12} and \cref{Fb_F11}, one deduces the
following relations
\begin{align*}
  \# F_{n,0} &= 2^{n-1},\\
  \# F_{n,b} &= 2^{n-1},\\
  \# F_{n,1} &= 2 \# F_{n-1,1} + \# F_{n-1,b}.
\end{align*}
which imply the following statement.

\begin{proposition}
  The number of elements in $F_n$ is the sequence \oeis{A045623}:
  \begin{equation}
    2^{n-2} (n+3) = 2,5,12,28,64,144,\dots
  \end{equation}
  for $n \geq 1$.
\end{proposition}




\begin{remark}
  Computing the Coxeter polynomials of the first few lattices in this
  familly, one observes that they have all their roots on the unit
  circle. This have been checked by computer up to the lattice with
  $3328$ elements. For example, for the poset $F_5$ of size $64$, the result
  is
  $\Phi_{1}^{2} \Phi_{2}^{4} \Phi_{6}^{4} \Phi_{7}
  \Phi_{23}^{2}$. One can see some very regular patterns in these
  Coxeter polynomials, when expressed as products of
  $[d]_x = (x^d-1)/(x-1)$ factors. This is a little further evidence
  that this roots-on-the-circle phenomenon could go on for larger
  cases.
\end{remark}





\section{Miscellany}

\subsection{A symmetry of colored $h$-polynomials}

\label{hpoly}

One can color the edges of the Hasse diagram of $\Dyck_n$ with two
colors as follows. When an edge corresponds to sliding a subpath to
its highest possible position, this edge is colored red. Other
edges are colored blue. For example, see \cref{fig_dyck4} where blue
edges are also dashed.

As we will see, this coloring is interesting because the generating
polynomial of incoming edges according to their colors has an
unexpected symmetry.

As a warm-up, let us start with the simpler generating series of incoming edges, not taking colors into account:
\begin{equation}
  A = \sum_{n \geq 0} \sum_{w \in \Dyck_n} x^{C(w)} t^n,
\end{equation}
where $C(w)$ is the number of elements covered by $w$.

Using the unique decomposition of Dyck paths into a list of blocks and \cref{facto_tous}, one gets
\begin{equation}
  \label{eq_h_A}
  A = \frac{1}{1-B},
\end{equation}
where $B$ is the similar sum restricted to block-indecomposable Dyck paths.

Then using the level-decomposition of block-indecomposables, one gets
\begin{equation}
  \label{eq_h_B}
  B = \frac{t}{1-x t A}.
\end{equation}
Indeed, the elements covered by $w = \level(w_1,\dots,w_k)$ either
come from replacing $w_i$ by some element that it covers, or from
sliding down a subpath of $w$ that ends somewhere in the final sequence of
letters $0$ in $w$. There are exactly $k$ such additional covered elements.

Together \eqref{eq_h_A} and \eqref{eq_h_B} imply an equation for $A$
which is exactly the well-known equation for the generating series of
Narayana polynomials. Note that the Narayana polynomials are the
$h$-vectors of the associahedra, namely the results of the same
counting of incoming edges for the Hasse diagram of the Tamari
lattices. This hints at a possible cellular structure of the Hasse
diagram of $\Dyck_n$, with the same $f$-vector as the
associahedron. This is left for a future study.

Let us now introduce a refined colored version of $A$:
\begin{equation}
  A = \sum_{n \geq 0} \sum_{w \in \Dyck_n} r^{C_r(w)} b^{C_b(w)} t^n,
\end{equation}
and the associated series $B$ restricted to
block-indecomposables. Here $C_r(w)$ and $C_b(w)$ count the red and blue
incoming edges at $w$.

The first equation \eqref{eq_h_A} holds unchanged, whereas \eqref{eq_h_B} must be slightly modified into
\begin{equation}
  B = \frac{t}{1-t(r+b(A-1))}.
\end{equation}
Indeed, the $k$ elements covered by $w = \level(w_1,\dots,w_k)$ that do
not come from an element covered by some $w_i$ can be either red or
blue. They are red exactly when $w_i$ is the empty Dyck path.

By elimination of $B$, one finds that $A$ satisfies the quadratic equation
\begin{equation}
  A^{2} t b + A t r - 2 A t b + A t -  t r + t b -  A + 1 = 0,
\end{equation}
from which one can deduce the global symmetry
\begin{equation}
  A - 1 = (A(1/r,b/r^2,r t) - 1)/r.
\end{equation}

This symmetry property can be stated as a simple
symmetry of the coefficient $A_n$ of $t^n$ in $A$:
\begin{equation}
  \forall n \geq 1 \quad A_n(r, r b) = r^{n-1} A_n(1/r, b/r).  
\end{equation}
The real meaning of this last symmetry is not clear for the moment. It
extends the usual symmetry of Narayana numbers.

\subsection{$m$-analogues}

One can easily define the same kind of variation for the $m$-Tamari
lattices \cite{bpr, mbmpr} instead of the Tamari lattices, using their similar
description by sliding subpaths in Dyck paths of slope $m$. These
posets do not seem to be very interesting, at least because their
numbers of intervals involve large prime numbers. For example, the
first few numbers of intervals for $m=2$ are given by
\begin{equation}
1, 1, 5, 36, 311, 3001, 31203.
\end{equation}

\subsection{Zeta polynomials and chains}

Let us now give a few simple experimental observations related to
chains and zeta polynomials in $\Dyck_n$. We have not tried to prove them.

The length of the longest chain in $\Dyck_n$ seems to be \oeis{A33638}, realized
between $\wmin$ and Dyck paths that ends with as many final
repetitions of $(1,0,0)$ as possible.

The first few values at $-1$ of the zeta polynomials of $\Dyck_n$ for
$n \geq 1$ are
\begin{equation}
  1, -1, 2, -5, 14, -42, 132, -429.
\end{equation}
One could guess that these should be (up to sign) the Catalan numbers.

The first few values at $-2$ of the zeta polynomials of $\Dyck_n$ for
$n \geq 1$ are
\begin{equation}
  1, -2, 7, -29, 131, -625, 3099, -15818.
\end{equation}
This coincides (up to sign) with the beginning of sequence
\oeis{A007852} that is counting antichains in rooted plane trees on
$n$ nodes.


\section{About the Tamari lattices}

\label{tamari}

Let us recall that the Tamari lattice \cite{tamari_friedman} of size
$n$ can be defined on the set of Dyck paths of size $n$ as the partial
order induced by transitive closure of some cover relations, namely
the exchange in any Dyck path of a letter $0$ (assumed to be followed
by $1$) with the subpath following it. This cover move is equivalent
to sliding this subpath by one step in the north-west direction. This
description is related in \cite[\S 2]{berbon} to the more classical
description using rotation on binary trees.

For a Dyck path $w \in \Dyck_n$, let the \defi{height sequence} of $w$
be $(h_1,h_2,\dots,h_n)$ where $h_i$ is the height in $w$ just after
the $i^{th}$ letter $1$. Note that repeated sliding to the north-west
of a given subpath $x$ always increase the same subsequence of the
height sequence.

\begin{lemma}
  \label{tamari_chain}
  Let $w$ be a Dyck path. Let $w'$ be obtained from $w$ by sliding
  once or several times the same subpath $x$ in the north-west
  direction. Then the interval $[w,w']$ in the Tamari lattice is a
  chain, in other words a total order, and all elements of $[w,w']$ are
  obtained from $w$ by sliding the subpath $x$.
\end{lemma}
\begin{proof}
  Only two kinds of Tamari cover moves can happen: either the subpath
  $x$ itself is slided, or another subpath $y$ is slided.

  Assume first that the latter happens, where $y$ is not contained in $x$. In
  that case, at least one element $h_i$ of the height sequence get
  increased, that is not in the subsequence modified when sliding $x$.
  Because the height sequence can never decrease, this element $h_i$ is
  still larger in $w'$ than it was in $w$, which is absurd.

  One can therefore assume that $y$ is contained in $x$. But sliding
  $x$ commutes with sliding such $y$.

  Let us pick an arbitrary chain of Tamari cover moves from $w$ to
  $w'$. By the commuting relation just explained, one can assume that
  all slidings of $x$ happen first in this chain.

  But then after performing these initial slidings of $x$, the first
  step of $x$ must have attained the position it will have in $w'$. So
  in fact at this point, the top element $w'$ has been reached
  already. Therefore our original chain only contains slidings of $x$.

  So the full interval $[w,w']$ is just made of a sequence of slidings of $x$.
\end{proof}

\bibliographystyle{alpha}
\bibliography{dextres}

\begin{thebibliography}{BMFPR11}

\bibitem[BB09]{berbon}
O.~Bernardi and N.~Bonichon.
\newblock Intervals in {C}atalan lattices and realizers of triangulations.
\newblock {\em J. Combin. Theory Ser. A}, 116(1):55--75, 2009.

\bibitem[BMFPR11]{mbmpr}
M.~Bousquet-M\'elou, {\'E}.~Fusy, and L.-F. Pr{\'e}ville-Ratelle.
\newblock The number of intervals in the {$m$}-{T}amari lattices.
\newblock {\em Electron. J. Combin.}, 18(2):Paper 31, 26, 2011.

\bibitem[BMJ06]{mbmj}
M.~Bousquet-M\'elou and A.~Jehanne.
\newblock Polynomial equations with one catalytic variable, algebraic series
  and map enumeration.
\newblock {\em J. Combin. Theory Ser. B}, 96(5):623--672, 2006.

\bibitem[BPR12]{bpr}
F.~Bergeron and L.-F. Pr\'eville-Ratelle.
\newblock Higher trivariate diagonal harmonics via generalized {T}amari posets.
\newblock {\em J. Comb.}, 3(3):317--341, 2012.

\bibitem[Cha07]{chap_slc}
F.~Chapoton.
\newblock Sur le nombre d'intervalles dans les treillis de {T}amari.
\newblock {\em S\'em. Lothar. Combin.}, 55:Art. B55f, 18, 2005/07.

\bibitem[Cha12]{chap_cats}
F.~Chapoton.
\newblock On the categories of modules over the {T}amari posets.
\newblock In {\em Associahedra, {T}amari lattices and related structures},
  volume 299 of {\em Prog. Math. Phys.}, pages 269--280.
  Birkh\"{a}user/Springer, Basel, 2012.

\bibitem[Fan18a]{fang3}
W.~Fang.
\newblock Planar triangulations, bridgeless planar maps and {T}amari intervals.
\newblock {\em European J. Combin.}, 70:75--91, 2018.

\bibitem[Fan18b]{fang2}
W.~Fang.
\newblock A trinity of duality: non-separable planar maps, {$\beta(1,0)$}-trees
  and synchronized intervals.
\newblock {\em Adv. in Appl. Math.}, 95:1--30, 2018.

\bibitem[FJN95]{freese}
R.~Freese, J.~Je{\v z}ek, and J.~B. Nation.
\newblock {\em Free lattices}, volume~42 of {\em Mathematical Surveys and
  Monographs}.
\newblock American Mathematical Society, Providence, RI, 1995.

\bibitem[FPR17]{fangpr}
W.~Fang and L.-F. Pr\'{e}ville-Ratelle.
\newblock The enumeration of generalized {T}amari intervals.
\newblock {\em European J. Combin.}, 61:69--84, 2017.

\bibitem[FT67]{tamari_friedman}
H.~Friedman and D.~Tamari.
\newblock Probl{\`e}mes d'associativit{\'e}: {U}ne structure de treillis finis
  induite par une loi demi-associative.
\newblock {\em J. Combinatorial Theory}, 2:215--242, 1967.

\bibitem[Gr{\"a}11]{gratzer}
G.~Gr{\"a}tzer.
\newblock {\em Lattice theory: foundation}.
\newblock Birkh{\"a}user/Springer Basel AG, Basel, 2011.

\bibitem[{Lad}07]{flipflop}
S.~{Ladkani}.
\newblock {Universal derived equivalences of posets of tilting modules}.
\newblock {\em ArXiv e-prints}, August 2007.

\bibitem[Len99]{lenzing}
H.~Lenzing.
\newblock Coxeter transformations associated with finite-dimensional algebras.
\newblock In {\em Computational methods for representations of groups and
  algebras ({E}ssen, 1997)}, volume 173 of {\em Progr. Math.}, pages 287--308.
  Birkh\"{a}user, Basel, 1999.

\bibitem[Lod11]{loday}
J.-L. Loday.
\newblock The diagonal of the {S}tasheff polytope.
\newblock In {\em Higher structures in geometry and physics}, volume 287 of
  {\em Progr. Math.}, pages 269--292. Birkh\H{a}user/Springer, New York, 2011.

\bibitem[MHPS12]{tamari_festschrift}
F.~M\"uller-Hoissen, J.~M. Pallo, and J.~Stasheff, editors.
\newblock {\em Associahedra, {T}amari lattices and related structures}, volume
  299 of {\em Progress in Mathematical Physics}.
\newblock Birkh\"auser/Springer, Basel, 2012.
\newblock Tamari memorial Festschrift.

\bibitem[Pal03]{pallo}
J.~M. Pallo.
\newblock Right-arm rotation distance between binary trees.
\newblock {\em Inform. Process. Lett.}, 87(4):173--177, 2003.

\bibitem[{Rog}18]{rognerud}
B.~{Rognerud}.
\newblock {Exceptional and modern intervals of the Tamari lattice}.
\newblock arXiv 1801.04097, January 2018.

\bibitem[San09]{san_arxiv}
S.~Saneblidze.
\newblock The bitwisted {C}artesian model for the free loop fibration.
\newblock {\em Topology Appl.}, 156(5):897--910, 2009.

\bibitem[San11]{san_note}
S.~Saneblidze.
\newblock On the homology theory of the closed geodesic problem.
\newblock {\em Rep. Enlarged Sess. Semin. I. Vekua Appl. Math.}, 25:113--116,
  2011.

\bibitem[Tut63]{tutte}
W.~T. Tutte.
\newblock A census of planar maps.
\newblock {\em Canad. J. Math.}, 15:249--271, 1963.

\end{thebibliography}

\end{document}